\documentclass[10pt,a4paper]{article}
\usepackage{epsfig}
\usepackage{amsmath}
\usepackage{amssymb}
\usepackage{amsfonts}
\usepackage{amsthm}
\usepackage{mathrsfs}
\usepackage{hyperref}
\usepackage{algorithm}
\usepackage{algorithmic}
\usepackage{blindtext}
\usepackage[utf8]{inputenc}
\usepackage{graphicx}
\usepackage{subfigure}
\usepackage[usenames, dvipsnames]{color}
\usepackage{caption}
\usepackage{booktabs}
\usepackage{varwidth}

\usepackage{algorithmic,float}
\newfloat{algorithm}{thp}{lop}
\floatname{algorithm}{Algorithm}
\newtheorem{theorem}{Theorem}[section]
 \newtheorem{lemma}[theorem]{Lemma}

 \newtheorem{example}[theorem]{Example}

\def\CC{{\mathbb C}}
\def\MatlabPackage{{\sf MultiParEig}}

\title{Subspace methods for 3-parameter eigenvalue problems}

\author{
Michiel~E.~Hochstenbach\thanks{%
Department of Mathematics and Computer Science, TU Eindhoven, PO Box 513, 5600 MB, The Netherlands,
{\tt www.win.tue.nl/$\sim$hochsten}. This author has been supported by an NWO Vidi research grant.}
\and
Karl Meerbergen\thanks{%
Department of Computer Science, KU Leuven, Celestijnenlaan 200A, 3001 Leuven, Belgium,
{\tt karl.meerbergen@cs.kuleuven.be}.}
\and
Emre Mengi\thanks{%
Department of Mathematics, Ko\c{c} University, Rumelifeneri Yolu, 34450 Sar{\i}yer-\.{I}stanbul, Turkey,
{\tt emengi@ku.edu.tr}. The research of this author was supported in part by the TUBITAK
(Scientific and Technological Research Council of Turkey) grant 115F585.}
\and
Bor~Plestenjak\thanks{%
IMFM and Faculty of Mathematics and Physics, University of Ljubljana, Jadranska 19,
SI-1000 Ljubljana, Slovenia, {\tt bor.plestenjak@fmf.uni-lj.si}. This author
was supported in part by the Slovenian Research Agency (grant P1-0294 and bilateral project ARRS-BI-TR/16-18-004 between Slovenia and Turkey).}
}

\begin{document}
\maketitle

\begin{abstract}
We propose subspace methods for 3-parameter eigenvalue problems.
Such problems arise when separation of variables is applied to separable boundary value problems;
a particular example is the Helmholtz equation in ellipsoidal and paraboloidal coordinates.
While several subspace methods for 2-parameter eigenvalue problems exist,
their extensions to three parameter setting seem to be challenging. An inherent difficulty is that, while
for 2-parameter eigenvalue problems we can exploit a relation to Sylvester
equations to obtain a fast Arnoldi type method, such a relation does not seem to exist when there are
three or more parameters. Instead, we introduce a subspace iteration method with projections
onto generalized Krylov subspaces that are constructed from scratch at every iteration using certain
Ritz vectors as the initial vectors. Another possibility is a Jacobi--Davidson type
method for three or more parameters, which we generalize from its 2-parameter counterpart.
For both approaches, we introduce a selection criterion for deflation that is based on the angles
between left and right eigenvectors. The Jacobi--Davidson approach is devised to
locate eigenvalues close to a prescribed target, yet it often also performs well when eigenvalues
are sought based on the proximity of one of the components to a prescribed target.
The subspace iteration method is devised specifically for the latter task. 
The proposed approaches are suitable especially for problems where the computation of several eigenvalues 
is required with high accuracy. Matlab implementations of both methods
have been made available in the package \texttt{MultiParEig} \cite{BorMC1}. \\[10pt]

\\[5pt]
\end{abstract}

\section{Introduction}
We consider an algebraic multiparameter eigenvalue problem of 
the form
\begin{equation} \label{eq::TEPproblemK}
\begin{array}{rcl}
A_{10} \, x_1&=&\lambda_1 \, A_{11} \, x_1+\cdots +\lambda_k \, A_{1k} \, x_1,\\[-0.4em]
&\vdots&\\[-0.4em]
A_{k0} \, x_k&=&\lambda_1 \, A_{k1} \, x_k+\cdots +\lambda_k \, A_{kk} \, x_k,
\end{array}
\end{equation}
where $A_{ij}\in\CC^{n_i\times n_i}$ are given matrices for $i=1,\ldots,k$ and $j=0,\ldots,k$.
We are looking for nonzero vectors $x_i\in\CC^{n_i}$ and a $k$-tuple $(\lambda_1,\ldots,\lambda_k)$
that satisfy \eqref{eq::TEPproblemK}. Such a $k$-tuple $(\lambda_1,\ldots,\lambda_k)$ is called an eigenvalue
and the tensor product $x_1\otimes \cdots \otimes x_k$ is called the corresponding eigenvector. For more
details on multiparameter eigenvalue problems, we refer to \cite{Atkinson}.

One possible source for such problems is the separation of variables; when applied to certain separable boundary value
problems, see, e.g., \cite{MoonSpencer, willbook}, we obtain a system of $k$ linear ordinary differential equations of the form
\begin{equation}
p_j(x_j) \, y_j''(x_j) + q_j(x_j) \, y_j'(x_j) + r_j(x_j) \, y_j(x_j) =
\sum_{\ell=1}^k \lambda_\ell \, s_{j\ell}(x_j) \, y_j(x_j),\label{eq::BDEproblem}
 \quad j=1,\ldots,k,
\end{equation}
where $x_j\in[a_j,b_j]$,
together with appropriate boundary conditions.
We are interested
in a $k$-tuple $(\lambda_1,\ldots,\lambda_k)$ and nontrivial functions $y_1,\ldots,y_k$
such that equations \eqref{eq::BDEproblem} and the boundary conditions are satisfied.
For more details on systems of the form \eqref{eq::BDEproblem} we refer to \cite{Atkinson2};
see also Section~\ref{sec:motivation}.

By discretizing \eqref{eq::BDEproblem} we obtain a problem of the form \eqref{eq::TEPproblemK}.
This approach is used in \cite{Calin3} to find numerical solutions for several separable boundary value problems
and improve previous results from the literature.
Specifically, spectral collocation is used in \cite{Calin3} for the discretization, which gives rise to
relatively small matrices and accurate results.
While several suitable numerical methods for the case $k=2$ exist, see, e.g., \cite{Calin3} and the references therein,
available feasible numerical methods for $k\ge 3$ are limited to problems with very small
matrices, which means that even by using spectral collocation,
we cannot obtain many accurate eigenvalues of \eqref{eq::BDEproblem}.
We introduce new variants of numerical methods for 3-parameter eigenvalue problems that exceed the above
limitations and can be applied to problems with larger matrices.
This allows us to solve efficiently and accurately several 3-parameter eigenvalue problems of the
form \eqref{eq::BDEproblem}, which we demonstrate in numerical examples.

Let $S_k$ denote the set of permutations of the set $\{1, \dots, k\}$, and let
${\rm sgn}(\sigma)$ be the sign of a permutation $\sigma\in S_k$.
By introducing the $k\times k$ operator determinants
\begin{equation}\label{eq:defn_Delta}
 \Delta_0   \; := \;
 \left|
 \begin{matrix}
 	A_{11} & \cdots & A_{1k}\cr
 	\vdots &  		& \vdots \cr
 	A_{k1} &  \cdots & A_{kk}
 \end{matrix}
 \right|_\otimes = \sum_{\sigma\in S_k}{\rm sgn}(\sigma) \
 A_{1\sigma_1}\otimes A_{2\sigma_2}\otimes \cdots \otimes A_{k\sigma_k},
\end{equation}
where $\otimes$ denotes the Kronecker product, and, similarly,
\begin{equation}\label{eq:defn_Deltab}
 \Delta_i  \; := \; \left|
 	\begin{matrix}
		A_{11} & \cdots & A_{1,i-1} & A_{10} & A_{1,i+1} & \cdots & A_{1k}\cr
			 \vdots & & \vdots &\vdots & \vdots & & \vdots \cr
 		A_{k1} & \cdots & A_{k,i-1} & A_{k0} & A_{k,i+1} & \cdots & A_{kk}
	\end{matrix}
 	\right|_\otimes
\end{equation}
for $i=1,\ldots,k$,
we obtain matrices $\Delta_0,\ldots,\Delta_k$ of size
$(n_1\cdots n_k)\times (n_1\cdots n_k)$.
If $\Delta_0$ is nonsingular, then the matrices $\Delta_0^{-1}\Delta_1,\ldots,
\Delta_0^{-1}\Delta_k$ commute, and \eqref{eq::TEPproblemK} is equivalent
to a system of generalized eigenvalue problems
\[
	\Delta_j \, z  =  \lambda_j \, \Delta_0 \, z, \quad\quad  j =1,\dots, k
\]
for $z=x_1\otimes \cdots \otimes x_k$ (for details, see, e.g., \cite{Atkinson}).
This relation enables one to use standard numerical methods for generalized eigenvalue
problems if the $\Delta$-matrices are not too large. However, when spectral methods
are used to discretize \eqref{eq::BDEproblem}, then in practice, even for $k=2$, the
$\Delta$-matrices might be so large that it is not efficient, or even not feasible,
to compute all of the eigenvalues. Fortunately, for various applications, the retrieval of
several eigenvalues closest to a prescribed target
is sufficient. In some other cases, eigenvalues $(\lambda_1, \dots, \lambda_k)$
such that a prescribed component among $\lambda_1,\dots, \lambda_k$ is close to a given target $\sigma$
are of interest. For instance, when we apply separation
of variables to the $k$-dimensional Helmholtz equation $\nabla^2 u+\omega^2 u=0$, usually
only one of the parameters $\lambda_1,\ldots,\lambda_k$ is related to the eigenfrequency $\omega$
(see Section~\ref{sec:motivation}
for more details). If we assume without loss of generality
that $\lambda_k$
is relevant to the problem and we are interested in first low-frequency modes for the Helmholtz equation, then
we are looking for eigenvalues with the smallest value of $|\lambda_k|$.

\subsection{Overview} 
Jacobi--Davidson type methods have been proposed for the 2-parameter eigenvalue problem in \cite{HKP, HPl02}
to compute a few eigenvalues closest to a prescribed target.
When eigenvalues $(\lambda_1, \dots, \lambda_k)$ with smallest $|\lambda_k|$ are sought,
subspace iteration or an Arnoldi iteration operating directly on $\Delta_k z = \lambda_k \Delta_0 z$
appears more appropriate. Such ideas have been explored well in the 2-parameter eigenvalue setting,
and applied for the solution of various separable boundary value problems \cite{KarlBor, Calin3}.
This success is mostly due to the fact that linear systems of the form $\Delta_2 w = \Delta_0 v$ for a given $v$
can be expressed as Sylvester equations involving the matricizations of the vectors $v$ and $w$,
and thus can be solved efficiently at a cost of $O(n_1^3 + n_2^3)$. An underlying difficulty is that such a Sylvester
equation representation is not known for the linear system $\Delta_k w = \Delta_0 v$
when $k \geq 3$.

The main contributions of this work are a Jacobi--Davidson method in Section~\ref{subsec:JD},
and an inexact subspace iteration method with Ritz projections in Section~\ref{sec:3SubIter} for 3-parameter
eigenvalue problems. The Jacobi--Davidson method is inspired by earlier works \cite{HKP, HPl02},
but new ingredients are also put in use. For instance,
a Newton-method based tensor Rayleigh quotient iteration is incorporated to speed up convergence.
Numerical experiments indicate that the proposed Jacobi--Davidson method is effective in extracting both the eigenvalues
closest to a prescribed target, and the eigenvalues $(\lambda_1,\lambda_2,\lambda_3)$ whose
$\lambda_3$ components are closest to a prescribed target. On the other hand,
inexact subspace iteration, which operates directly on the generalized eigenvalue problem
$\Delta_3 z = \lambda_3 \Delta_0 z$, is tailored to compute eigenvalues with their
$\lambda_3$ components closest to a prescribed target. Instead of solving
 a linear
system of the form $\Delta_3 w = \Delta_0 v$ for the unknown $w$, it projects the
full problem onto certain generalized Krylov subspaces that are restarted
at every iteration with selected Ritz vectors. We especially aim at problems where the computation of
several extreme eigenvalues is required with high accuracy. Both of the proposed Jacobi--Davidson method
and inexact subspace iteration are well-suited to deal with such problems.

\subsection{Outline}
We start with 
two particular applications giving rise to 3-parameter eigenvalue problems
in Section~\ref{sec:motivation}; this is followed by a brief review of subspace iteration approaches for the 2-parameter case
in Section~\ref{sec:2parameter}. In particular, efficient solutions of the linear system
$\Delta_2 w = \Delta_0 v$ with or without projections via their Sylvester equation characterization facilitate
these approaches.

The main body 
is Section~\ref{sec:3par}, which
introduces iterative methods for the
extraction of a few targeted eigenvalues of a 3-parameter eigenvalue problem. A Jacobi--Davidson
method is proposed in Section~\ref{subsec:JD}. The difficulty intrinsic to applying a Krylov subspace
method directly to $\Delta_3 z = \lambda_3 \Delta_0 z$ is pointed out in Section~\ref{subs:fullKrylov}.
Consequently, in Section~\ref{sec:3SubIter_v1},
a subspace iteration method that does not work on the full linear systems, but rather solves their projections
onto Krylov subspaces, is described. The downside of this
approach is that in
every iteration it requires low-rank third-order tensor approximations
for the solutions of the linear systems. Finally, an efficient Krylov subspace based subspace iteration is proposed
in Section~\ref{sec:3SubIter}, which
employs the projection ideas in Section~\ref{sec:3SubIter_v1},
but
removes the need for low-rank tensor approximations.

Section~\ref{sec:NumResults} is devoted to extensive numerical experiments. In particular, we illustrate how the
proposed Jacobi--Davidson and subspace iteration methods perform on the 3-parameter eigenvalue problems
resulting from the applications in Section~\ref{sec:motivation}, as well as on a random synthetic example.

\section{Motivation}\label{sec:motivation}
We give two applications
that lead to 3-parameter eigenvalue problems of the form \eqref{eq::BDEproblem}.
They concern the separation of variables applied to the Helmholtz equation
\begin{equation}\label{eq:helm}
\nabla^2 u + \omega^2 u = 0
\end{equation}
in ellipsoidal and paraboloidal coordinates.

\subsection{Ellipsoidal wave equations}\label{sec:motive_ellip_wave}
If we aim to compute eigenfrequencies of an ellipsoidal body
with a fixed boundary, then we have to solve the Helmholtz equation \eqref{eq:helm}
over the ellipsoid
\[
\Omega := \{ \: (x,y,z) \in {\mathbb R}^3 \;\; | \;\; (x/x_0)^2 + (y/y_0)^2 + (z/z_0)^2 \: \leq \: 1 \}
\]
subject to the Dirichlet boundary condition
$
 u|_{\partial \Omega} = 0.
$
Here, $x_0, y_0, z_0$ correspond to the radii of the semi-axes of the ellipsoid and satisfy $z_0 > y_0 > x_0 > 0$.
A numerical approach has been proposed in \cite{Willatzen2005},
see also \cite{Levitina, Calin3}; here we give an outline of how it leads to a 3-parameter eigenvalue problem.

The Helmholtz equation is separable in ellipsoidal coordinates $(\xi_1, \xi_2, \xi_3)$ \cite{MoonSpencer},
a natural choice for the region $\Omega$. Formally,
there exist functions $X_1(\xi_1), X_2(\xi_2)$, $X_3(\xi_3)$ such that the solution can be written as
\[
 u(x(\xi_1, \xi_2, \xi_3), \, y(\xi_1, \xi_2, \xi_3), \, z(\xi_1, \xi_2, \xi_3))
 \; = \; X_1(\xi_1) \, X_2(\xi_2) \, X_3(\xi_3).
\]
Exploiting the separability property above and expressing the Helmholtz equation in ellipsoidal coordinates, we obtain three
ordinary differential equations
\[
 t_j\, (t_j - 1) (t_j - c) \, \widetilde{X}_j'' +
 \tfrac{1}{2} (3t_j^2 - 2(1+c)t_j + c) \, \widetilde{X}_j' +
 (\lambda + \mu t_j + \eta t_j^2) \, \widetilde{X}_j = 0,
 \quad j = 1,2,3,
\]
where $c = a^2 / b^2$, $a = (z_0^2 - x_0^2)^2$, $b = (z_0^2 - y_0^2)^2$, $t_j = \xi_j^2 / b^2$,
$\widetilde{X}_j(t_j) := X_j (\xi_j(t_j))$,
and the elliptical coordinates satisfy $z_0 > \xi_1 > a > \xi_2 > b > \xi_3 > 0$.
The three
differential equations are coupled by the scalars $\lambda, \mu$, $\eta$, but
only $\eta = \omega^2 b^2/4$ is related to the eigenfrequency $\omega$.
The function $\widetilde{X}_j(t_j)$ above is of the form
\[
 \widetilde{X}_j(t_j) \: = \: t_j^{\rho/2} \, (t_j-1)^{\sigma/2} \, (t_j-c)^{\tau/2} \, F_j(t_j)
\]
where $F_j(t_j)$ is an integral function of $t_j$, and $\rho, \sigma, \tau$ can take values $0$ or $1$.
For each one of the eight possible configurations for $(\rho,\sigma,\tau)$,
we deduce the system of ordinary differential equations
\begin{equation}\label{eq:main_ODE_sys}
 t_j(t_j - 1) (t_j - c)\,F_j'' + \tfrac{1}{2} (k_2 t_j^2 - 2k_1 t_j + k_0)\,F_j'
 + (\lambda - \lambda_0 + (\mu + \mu_0) t_j + \eta t_j^2)\,F_j = 0,
 \quad j = 1,2,3,
\end{equation}
with
\begin{equation*}
 \begin{split}
 \lambda_0 = \tfrac{1}{4} \left[ (\rho + \tau)^2 + (\rho + \sigma)^2 c \right], \quad \mu_0 =
 \tfrac{1}{4} (\rho + \sigma + \tau)(\rho + \sigma + \tau + 1), \quad \; \\
 k_0 = (2 \rho + 1)c, \quad k_1 = (1+\rho)(1+c) + \tau + \sigma c, \quad k_2 = 2(\rho + \sigma + \tau) + 3.
 \end{split}
\end{equation*}
The boundedness conditions at singular points and Dirichlet condition on the boundary of the ellipsoid
give rise to the following boundary conditions:
\begin{align*}
F_1(z_0^2/b^2)&=0,\\
(k_2c^2-k_1c+k_0) \, F_j'(c)+2 \, (\lambda-\lambda_0+(\mu+\mu_0)c+\eta c^2) \, F_j(c)&=0\quad {\rm for}\quad j=1,2,\\
(k_2-2k_1+k_0) \, F_j'(1)+2 \, (\lambda-\lambda_0+\mu+\mu_0+\eta) \, F_j(1)&=0\quad {\rm for}\quad j=2,3,\\
k_0 \, F_3'(0)+2 \, (\lambda-\lambda_0) \, F_3(0)&=0.
\end{align*}
This example will be solved numerically in Section \ref{sec:result_ellip_wave} more accurately than in \cite{Calin3} as
the new methods can deal with larger matrices coming from finer discretizations.

\subsection{Baer wave equations}\label{sec:motive_Baer_wave}

Helmholtz equation \eqref{eq:helm} is also separable in paraboloidal coordinates $(\xi_1,\xi_2,\xi_3)$, which are
related to the Cartesian coordinates by (see, e.g., \cite{Duggen, MoonSpencer})
\begin{align*}
x^2&=4(c-b)^{-1}\,(b-\xi_1)\,(b-\xi_2)\,(b-\xi_3),\\
y^2&=4(b-c)^{-1}\,(c-\xi_1)\,(c-\xi_2)\,(c-\xi_3),\\
z&=\xi_1+\xi_2+\xi_3-b-c,
\end{align*}
where $-\infty<\xi_1<c<\xi_2<b<\xi_3<\infty$ and $c<b$ are the parameters of the paraboloidal coordinate system.
A constant surface $\xi_1=\gamma$, where $\gamma<c$, represents an upward opening elliptic paraboloid
which intersects the $z$-axis at $z=\gamma$, while a constant surface $\xi_3=\beta$, where $b<\beta$, represents a
downward opening elliptic paraboloid which intersects the $z$-axis at $z=\beta$.

In Section~\ref{sec:result_Baer_wave}, we will consider the
solution of the Helmholtz equation with a fixed boundary on a domain bounded by the two elliptic paraboloids
$\gamma=0$ and $\beta=5$, as well as for the choices of $c = 1$ and $b = 3$, see
Figure \ref{fig:paraboloid}.

\begin{figure}[htb]
\begin{center}
\includegraphics[scale=0.23]{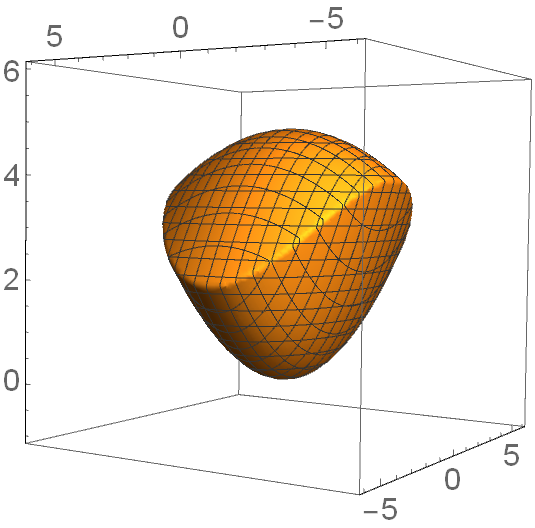}\qquad
\includegraphics[scale=0.23]{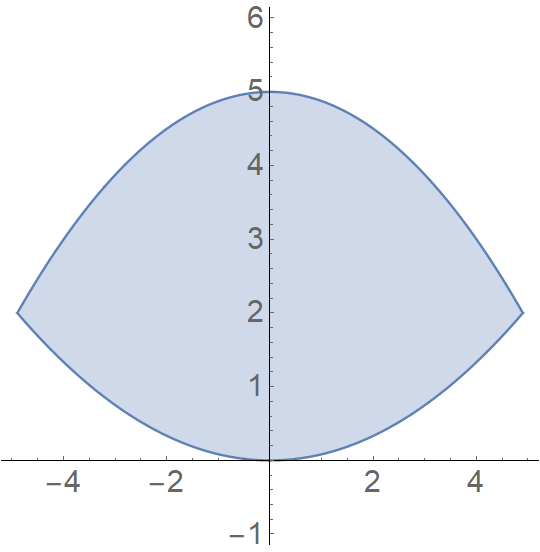}\qquad
\includegraphics[scale=0.23]{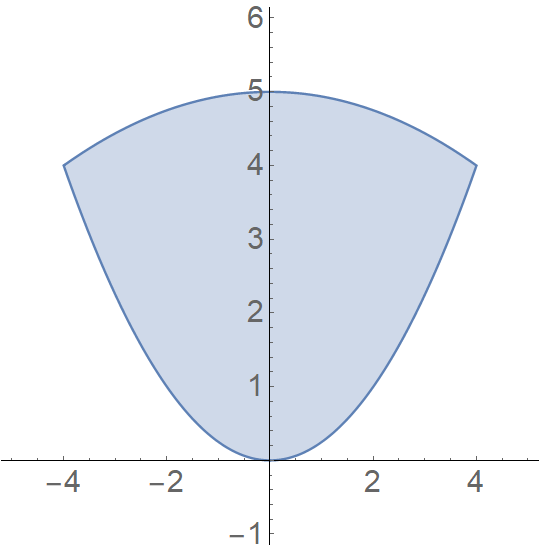}
\end{center}
\vspace{-0.7em}
\caption{Region bounded in paraboloidal coordinates by elliptical paraboloids $\xi_1=0$ and $\xi_3=5$ (left),
its intersection with $xz$-plane (middle), and intersection with $yz$-plane (right).}
\label{fig:paraboloid}
\vspace{-1em}
\end{figure}

We use separation of variables. The solution of \eqref{eq:helm} has the form
$u=X_1(\xi_1)\,X_2(\xi_2)\,X_3(\xi_3)$ \cite{MoonSpencer}, where $X_1,X_2,X_3$ satisfy the system
of Baer wave differential equations given by
\begin{equation}\label{eq:baer}
(\xi_j-b)(\xi_j-c)\,X_j''+\tfrac{1}{2}(2\xi_j-(b+c))\,X_j'+(\lambda +\mu \xi_j +\eta \xi_j^2)\,X_j = 0,
\quad\quad j = 1,2,3,
\end{equation}
and $\xi_1, \xi_2, \xi_3$ are such that $\gamma<\xi_1<c<\xi_2<b<\xi_3<\beta$.
In the equations above, $\eta=\omega^2$ is related to the eigenfrequency, whereas parameters
$\lambda$ and $\mu$ result from the separation. Equation \eqref{eq:baer} has regular singularities
at $b$ and $c$, and an irregular singularity at infinity. The exponents at the finite singularities are
$0$ and $1/2$. Therefore, it is possible to write the solution of \eqref{eq:baer} as
\begin{equation}\label{eq:baerXF}
X_i(\xi_i)=(\xi_i-b)^{\rho/2}\,(\xi_i-c)^{\sigma/2}\,F_i(\xi_i),
\end{equation}
where $F_i(\xi_i)$ is an integral function of $\xi_i$, and
$\rho,\sigma$ can be either $0$ or $1$ leading to four possible configurations.

For a particular $(\rho,\sigma)$ configuration,
by plugging \eqref{eq:baerXF} into \eqref{eq:baer}, we obtain the system
\begin{equation}\label{eq:baerF}
(\xi_j-b)(\xi_j-c)\,F_j''+\tfrac{1}{2}(k_1 \xi_j - k_0)\,F_j'+(\lambda-\lambda_0 +\mu \xi_j +\eta \xi_j^2)\,F_j=0,\quad j=1,2,3
\end{equation}
of differential equations, where
\[
k_1 = 2(1+\rho+\sigma),\quad
k_0 = (1+2\sigma)b+(1+2\rho)c,\quad
\lambda_0=-\,\tfrac{1}{4}(\rho+\sigma+2\rho\sigma).
\]
The boundedness conditions at singular points, and the Dirichlet condition on the boundary of the domain yield the following boundary conditions:
\begin{equation}\label{baer:bc}
\begin{split}
F_1(\gamma) & =0,\\
\tfrac{1}{2}(k_1 c-k_0) \, F_j'(c)+(\lambda-\lambda_0+\mu c+\eta c^2) \, F_j(c) & =0\quad {\rm for}\quad j=1,2,\\
\tfrac{1}{2}(k_1 b-k_0) \, F_j'(b)+(\lambda-\lambda_0+\mu b+\eta b^2) \, F_j(b) & =0\quad {\rm for}\quad j=2,3,\\
F_3(\beta) & =0.
\end{split}
\end{equation}
We will present some numerical experiments with these examples in Section~\ref{sec:NumResults}.

\section{Two parameters}\label{sec:2parameter}
In this section, we consider \eqref{eq::TEPproblemK} for the case $k=2$, but, to ease the notation,
set  $A_{j} := A_{j0}$, $B_{j} := A_{j1}$, $C_{j} := A_{j2}$ for $j = 1,2$, as well as
$\lambda := \lambda_1$, $\mu := \lambda_2$. A quick overview of the ideas in \cite{KarlBor} is presented next.
As we shall see in the subsequent section, most of these ideas for the two parameter case cannot be
generalized to more than two parameters.

Recall that if $\Delta_0=B_1\otimes C_2-C_1\otimes B_2$ is nonsingular,
then the two parameter problem at hand is equivalent to a coupled pair of generalized eigenvalue problems
$\Delta_1 \, z=\lambda \, \Delta_0 \, z$ and $\Delta_2 \, z=\mu \, \Delta_0 \, z$ for $z=x_1\otimes x_2$,
where $\Delta_1=A_1\otimes C_2-C_1\otimes A_2$ and $\Delta_2=B_1\otimes A_2-A_1\otimes B_2$.
Suppose that we are looking for the eigenvalues $(\lambda,\mu)$ with the smallest
value of $|\mu|$, and let us assume that $n_1n_2$ is so large that we cannot efficiently compute all eigenvalues
of the generalized eigenvalue problem
\begin{equation}\label{eq:D2D0}
 	\Delta_2 \, z=\mu \, \Delta_0 z.
\end{equation}
Next, we discuss two alternative numerical approaches for this setting: Krylov subspace methods and a subspace
iteration.

\subsection{Krylov subspace methods}
If $n_1n_2$ is not too large, then we can apply a Krylov subspace method to \eqref{eq:D2D0},
for instance the implicitly restarted Arnoldi \cite{SVDV} or the Krylov--Schur method \cite{Stewart}. 
As we are interested in the smallest values
of $|\mu|$, we want to build an orthogonal basis for the Krylov subspace ${\cal K}_k(\Delta_2^{-1}\Delta_0,v_0)$,
which means that in each step we have to compute a matrix-vector product with the matrix $\Delta_2^{-1}\Delta_0$.
The key observation to perform this multiplication efficiently is its connection with a Sylvester equation. Namely,
the expression $w =\Delta_2^{-1}\Delta_0 v$ can be rearranged as
\begin{equation}\label{eq:Kronecker_lin_sys}
 (B_1\otimes A_2-A_1\otimes B_2)\,w =(B_1\otimes C_2-C_1\otimes B_2)\,v.
\end{equation}
Using the vectorization operator
 \[ X :=
 [X_1 \ \cdots \ X_q]
\in {\mathbb C}^{p\times q},\ \: X_1, \dots, X_q \in {\mathbb C}^p \;\;
 \mapsto \;\;
 {\rm vec}(X)
 :=
 [X_1^T \ \cdots \ X_q^T]^T
 \in {\mathbb C}^{pq},
 \]
and the identity $(B\otimes A)\,{\rm vec}(X) \: = \: {\rm vec}(AXB^T)$, we can write the linear system
in \eqref{eq:Kronecker_lin_sys} as
\[
 A_2 W B_1^T - B_2 W A_1^T = C_2 V B_1^T - B_2 V C_1^T,
\]
where $W$ and $V$ are matrices such that ${\rm vec}(W)=w$ and ${\rm vec}(V)=v$. If we assume that $B_1$ and $B_2$
are nonsingular, then this is equivalent to the Sylvester equation
\begin{equation}\label{eq:sylvd2d0}
 B_2^{-1}A_2 W - W A_1^TB_1^{-T} \; = \; M ,
\end{equation}
where $M = B_2^{-1} C_2 V - VC_1^T B_1^{-T}$. As the Sylvester equation in \eqref{eq:sylvd2d0} can be solved in
${\cal O}(n_1^3+n_2^3)$ operations using, e.g., the Bartels--Stewart method \cite{Bartels},
this is much more efficient than forming $\Delta_2$ and $\Delta_0$ explicitly, and then solving
$\Delta_2w=\Delta_0 v$,
which typically requires ${\cal O}(n_1^3n_2^3)$ operations.

If $n_1 n_2$ is even larger, we can neither store many vectors from
${\cal K}_k(\Delta_2^{-1}\Delta_0,v_0)$ fully
nor perform exact computations with them efficiently. In the limit
(as we keep multiplying
with $\Delta_2^{-1} \Delta_0$) $v$ and $w$ in \eqref{eq:Kronecker_lin_sys} are collinear to the dominant eigenvector
of the form $z=x\otimes y$ and the corresponding matrices $V$ and $W$ have both rank one.
Hence, the right-hand side $M = B_2^{-1} C_2 V - VC_1^T B_1^{-T}$ of the Sylvester equation \eqref{eq:sylvd2d0} is nearly
of rank two at later iterations, whereas its solution $W$
has almost rank one.
In this case, it is possible to benefit from an approximate low-rank solver
for the Sylvester equation, see, e.g., \cite{KarlBor} that makes use of an approximate
Krylov subspace solver due to Hu--Reichel \cite{Hu1992}.

The main idea of the Hu--Reichel method is as follows. Suppose that
the Sylvester equation
\begin{equation}\label{eq:general_Syl}
 AX-XB=C
\end{equation}
is such that $C$ has low rank, and additionally suppose that the solution $X$
is expected to have low rank (in practice it is enough that both $C$ and $X$ are close to low-rank matrices).
If $C\approx FG^T$, where $F$ and $G$ have a few
columns, then matrices $Q_A$ and $Q_B$ whose columns form orthonormal bases for the Krylov subspaces
${\cal K}_r(A,F)$ and ${\cal K}_r(B^T,G)$ are built. An approximate solution for \eqref{eq:general_Syl} is then
given by $X=Q_AYQ_B^H$, where the matrix $Y$ is the solution of the small-scale
projected Sylvester equation
\[
 Q_A^HAQ_AY-YQ_B^HBQ_B=Q_A^HCQ_B.
\]

For further details and various other numerical approaches for large-scale Sylvester equations, we refer to the survey
paper \cite{Simoncini2016} and the references therein.

\subsection{Subspace iteration}
The subspace iteration starts with a matrix $Z_0\in\CC^{n_1n_2\times p}$,
such that $Z_0^HZ_0=I$, where $p\ll n_1n_2$.
In each step, first $p$ linear systems $\Delta_2 W_{k+1}=\Delta_0 Z_k$ are solved, then the columns
of $W_{k+1}$ are orthonormalized into $Z_{k+1}$. As $k$ goes to infinity, $W_{k+1}^H Z_k$ under mild conditions converges
to an upper triangular matrix with eigenvalues of the pencil $(\Delta_2, \Delta_0)$ on its diagonal.
As discussed in the previous subsection, the linear systems $\Delta_2 W_{k+1}=\Delta_0 Z_k$ can
be expressed as a set of $p$ Sylvester equations, each one of which generically posses a
low-rank structure when converging.
Consequently, instead of working with full vectors in the
columns of the matrices $Z_k$ and $W_k$, we rather use their low-rank approximations.
We express the columns of $Z_k$ as $z_i^{(k)}={\rm vec}(U_kD_i^{(k)}V_k^T)$ for $i = 1,\dots,p$,
where $U_k\in\CC^{n_1\times \ell}$ and $V_k\in\CC^{n_2\times \ell}$ have orthonormal columns and
$D_i^{(k)}$ is an $\ell\times \ell$ core matrix, where $\ell\ge p$.
Instead of solving $\Delta_2 w_i^{(k+1)}=\Delta_0 z_i^{(k)}$
exactly, we solve this only approximately and obtain a low-rank approximation for the matricization of
$w_i^{(k+1)}$ by means of the Hu--Reichel method. Specifically, we search for $w_i^{(k+1)}$
in the space
\begin{equation}\label{eq:abg}
{\cal K}_r(B_1^{-1} A_1,G)\otimes {\cal K}_r(B_2^{-1}A_2,F)
\end{equation}
for a modest $r$,
where
$F =
 [\begin{matrix}
 B_2^{-1}C_2U_k & U_k
 \end{matrix}]$
and
$G =
 [\begin{matrix}
 B_1^{-1}C_1V_k & V_k
 \end{matrix}]$ (i.e., setting $C \approx F G^T$ in
the Sylvester equation
\eqref{eq:general_Syl}
equal to $M$ as in \eqref{eq:sylvd2d0} yields these choices for $F$ and $G$).
For each $w_i^{(k+1)}$, $i=1,\ldots,p$, we solve a small projected Sylvester equation.
If the columns of $\widetilde V_{k}$ and $\widetilde U_{k}$ form orthonormal bases for
${\cal K}_r(B_1^{-1} A_1,G)$ and ${\cal K}_r(B_2^{-1} A_2,F)$, respectively, then
$w_i^{(k+1)}={\rm vec}(\widetilde U_{k}Y_i^{(k)}\widetilde V_{k}^T)$ for some $Y_i^{(k)}$.
For the construction of $U_{k+1}$ and $V_{k+1}$ we note that all vectors $w_1^{(k+1)}, \dots, w_p^{(k+1)}$
in the next step lie in \eqref{eq:abg}. This inspired a new method
in \cite{KarlBor} called subspace iteration with Arnoldi expansion.
The essential idea is to compute matrices $\widetilde V_k, \widetilde U_k$ whose columns form orthonormal
bases for ${\cal K}_r(B_1^{-1} A_1,G)$, ${\cal K}_r(B_2^{-1} A_2,F)$, and then to compute the Ritz values
with the smallest values of $|\tau|$
as well as the Ritz vectors from the projected small-scale 2-parameter eigenvalue problem
\begin{align*}
{\widetilde V_k}^HA_1\widetilde V_{k}\,c &=
\sigma \, \widetilde V_{k}^HB_1\widetilde V_{k}\,c + \tau\,\widetilde V_{k}^HC_1\widetilde V_{k}\,c \\
\widetilde U_k^HA_2\widetilde U_k\,d &=
\sigma \, \widetilde U_k^HB_2\widetilde U_k\,d + \tau\,\widetilde U_{k}^HC_2
\widetilde U_{k}\,d.
\end{align*}
From $\ell$ such Ritz vectors, which are all
decomposable, we form the new subspaces ${\rm span}\{U_{k+1}\}$ and ${\rm span}\{V_{k+1}\}$ for the next step.

\section{Three parameters}\label{sec:3par}
Let us now focus on 3-parameter eigenvalue problems, which are of the form \eqref{eq::TEPproblemK} for $k = 3$.
Similarly to the previous section, to ease the notation, we let
$A_j := A_{j0}$, $B_j := A_{j1}$, $C_j := A_{j2}$, $D_j := A_{j3}$
for $j = 1, 2, 3$, where $A_{jk}$ are as in \eqref{eq::TEPproblemK}, and $\lambda := \lambda_1$,
$\mu := \lambda_2$, $\eta := \lambda_3$.
In this 3-parameter eigenvalue setting, we are seeking the eigenvalues with the smallest values of $|\eta|$.
They correspond to the eigenvalues of the generalized eigenvalue problem
\begin{equation}\label{eq:3pdelta}
	\Delta_3 z=\eta \, \Delta_0 z
\end{equation}
with the smallest values of $|\eta|$, provided $\Delta_0$ is nonsingular, where $\Delta_0$ and $\Delta_3$ denote
matrices involving third tensors defined by (\ref{eq:defn_Delta}) and
(\ref{eq:defn_Deltab}). If such an eigenvalue is simple, then
the corresponding eigenvector $z$ is decomposable and can be expressed as $z=x_1\otimes x_2\otimes x_3$.

\subsection{Using full $\Delta$-matrices} The first option is to explicitly form the matrices
$\Delta_0$ and $\Delta_3$, and then use the QZ algorithm (or any other numerical method) to compute the
eigenvalues of \eqref{eq:3pdelta}. As the size of the matrices $\Delta_0$ and $\Delta_3$ is
$n_1n_2n_3\times n_1n_2n_3$, this is efficient only when $n_1 n_2 n_3$ is small. This approach becomes
prohibitively expensive even for modest values of $n_1, n_2$, $n_3$.

\subsection{Jacobi--Davidson type method}\label{subsec:JD}
Methods of Jacobi--Davidson type have been developed for 2-parameter eigenvalue problems in \cite{HKP, HPl02, HP08}.
As long as we are able to solve a small projected problem efficiently, the method can be generalized
to multi-parameter eigenvalue problems with three or more parameters. Inspired by its 2-parameter counterpart
in \cite{HPl02}, we give a brief description of a Jacobi--Davidson type
method for a 3-parameter eigenvalue problem in Algorithm~\ref{alg:jd}. In the description, $\|r_j\|$
represents the 2-norm of the residual $r_j$ and {\sf rgs} stands for repeated Gram--Schmidt orthogonalization.

\begin{algorithm}[h]
\begin{algorithmic}[1]
\STATE Choose initial matrices $U_j^{(0)}\in\CC^{n_j\times \ell}$ with orthonormal columns for $j=1,2,3$.
\FOR {$k=0, 1, \ldots$}
\STATE Compute appropriate Ritz value $(\sigma,\tau,\psi)$ and vector $U_1^{(k)}s_1 \otimes U_2^{(k)}s_2\otimes U_3^{(k)}s_3$
from the projected 3-parameter eigenvalue problem\label{lin1:extraction}
\begin{align*}
{U_1^{(k)}}^HA_1U_1^{(k)}s_1 &= \sigma\,{U_1^{(k)}}^HB_1U_1^{(k)} s_1+ \tau\,{U_1^{(k)}}^HC_1U_1^{(k)} s_1 +
 \psi\,{U_1^{(k)}}^HD_1U_1^{(k)} s_1, \\
{U_2^{(k)}}^HA_2U_2^{(k)}s_2 &= \sigma\,{U_2^{(k)}}^HB_2U_2^{(k)} s_2+ \tau\,{U_2^{(k)}}^HC_2U_2^{(k)} s_2 +
 \psi\,{U_2^{(k)}}^HD_2U_2^{(k)} s_2, \\
{U_3^{(k)}}^HA_3U_3^{(k)}s_3 &= \sigma\,{U_3^{(k)}}^HB_3U_3^{(k)} s_3+ \tau\,{U_3^{(k)}}^HC_3U_3^{(k)} s_3 +
 \psi\,{U_3^{(k)}}^HD_3U_3^{(k)} s_3.
\end{align*}\vspace{-1em}
\STATE Compute the residual $r_j=(A_j-\sigma B_j-\tau C_j-\psi D_j)\,u_j$, where $u_j=U_j^{(k)}s_j$, for $j=1,2,3$.\smallskip
\IF {$(\|r_{1}\|^2+\|r_{2}\|^2+\|r_{3}\|^2)^{1/2}\le \delta$}\label{lin1:delta}
\STATE Refine the Ritz pair by applying $t\ge 0$ steps of the TRQI and update the residuals.\label{lin1:refine}
\STATE If the refined pair satisfies the selection criterion and $(\|r_{1}\|^2+\|r_{2}\|^2+\|r_{3}\|^2)^{1/2}\le \varepsilon$,
then extract the eigenpair and compute the corresponding left eigenvector.
\label{lin1:convergence}
\ELSE
\STATE Solve for $j=1,2,3$ (approximately or exactly) the correction equation\label{lin1:correq}
\begin{equation}\label{eq:correq}
(I-u_ju_j^H)(A_j-\sigma B_j-\tau C_j-\psi D_j)\,v_j=-r_j, \qquad v_j \perp u_j.
\end{equation}
\vspace{-1em}
\STATE Expand $U_j^{(k+1)}={\sf rgs}(U_j^{(k)},v_j)$ for $j=1,2,3$.\label{lin1:exp}
\STATE If the dimension of $U_j^{(k+1)}$ is too large, construct new\label{lin1:restart}
 $U_j^{(k+1)}\in\CC^{n_j\times \ell}$ for $j=1,2,3$.
\ENDIF
\ENDFOR
\end{algorithmic}
\caption{\emph{Jacobi--Davidson method for the
3-parameter eigenvalue problem.} \\
In the algorithm, $\ell$ denotes the size of the subspace after a restart,
$\varepsilon$ is used in the convergence criterion for an eigenvalue, and $\delta > \varepsilon$ is used to decide whether
a Ritz pair is a candidate for TRQI refinement.}
\label{alg:jd}
\end{algorithm}

In Algorithm~\ref{alg:jd} we extract one eigenpair at a time. A small projected 3-parameter eigenvalue problem is solved
in each step. If an eigenpair has converged, then we keep the current subspace, as it may lead to other
eigenvalues. Otherwise, we expand the subspace with the addition of a vector that satisfies
the correction equation \eqref{eq:correq} in line~\ref{lin1:exp}, where we apply repeated Gram--Schmidt
orthogonalization. In what follows, we spell out some of the
important details of the algorithm. \smallskip

\noindent
\textit{\textbf{Targeting.}} Depending on the application, a prescribed eigenvalue target can be
either a point $(\lambda_0, \mu_0, \eta_0)$ or a plane, e.g., $\eta=0$.
For instance, if we take $(0,0,0)$ as the target, then we search 
for eigenvalues
with the minimal value of $|\lambda|^2+|\mu|^2+|\eta|^2$.
In line~\ref{lin1:extraction}, we select a particular Ritz value $(\sigma,\tau,\psi)$ that is closest to the
target and satisfies an additional selection criterion described below.\medskip

\noindent
\textit{\textbf{Selection Criterion.}}
The purpose of the selection criterion is to prevent convergence to an eigenvalue that has already been detected.
The criterion is based on the following lemma, which is a straightforward generalization of its 2-parameter
counterpart (see \cite{HKP}).\smallskip

\begin{lemma}
Let $(\lambda_1,\mu_1,\eta_1)\ne (\lambda_2,\mu_2,\eta_2)$ be different
eigenvalues of the 3-parameter eigenvalue problem
such that $(\lambda_1,\mu_1,\eta_1)$ is a simple eigenvalue with the
right eigenvector $x_1^{(1)} \otimes x_2^{(1)} \otimes x_3^{(1)}$ and the
left eigenvector $y_1^{(1)} \otimes y_2^{(1)} \otimes y_3^{(1)}$. If $y_1^{(2)} \otimes y_2^{(2)} \otimes y_3^{(2)}$
is a left eigenvector corresponding to $(\lambda_2,\mu_2,\eta_2)$, then\vspace{0.2em}
\begin{enumerate}
 \item[\bf (i)] $(y_1^{(1)} \otimes y_2^{(1)} \otimes y_3^{(1)} )^H\Delta_0(x_1^{(1)} \otimes x_2^{(1)} \otimes x_3^{(1)} ) \ne 0$,
 and\vspace{0.2em}
 \item[\bf (ii)] $(y_1^{(2)} \otimes y_2^{(2)} \otimes y_3^{(2)} )^H\Delta_0(x_1^{(1)} \otimes x_2^{(1)} \otimes x_3^{(1)} ) = 0$.
\end{enumerate}
\end{lemma}

\medskip

Let $(\lambda_q,\mu_q,\eta_q)$ be the eigenvalues
that are already extracted along with the corresponding left and right eigenvectors $y_1^{(q)}\otimes y_2^{(q)}\otimes y_3^{(q)}$ and
$x_1^{(q)}\otimes x_2^{(q)}\otimes x_3^{(q)}$
for $q=1,\ldots,m$. The selection criterion outlined next is based on these eigenvectors.
In line~\ref{lin1:extraction} of Algorithm~\ref{alg:jd}, we select a Ritz value
such that the corresponding Ritz vector $u_1\otimes u_2\otimes u_3$, with $u_j = U_j^{(k)} s_j$ for $j = 1,2,3$,
satisfies
\begin{equation}\label{eq:selcrit}
 \max_{q=1,\ldots,m}
 \frac{\big|(y_1^{(q)}\otimes y_2^{(q)}\otimes y_3^{(q)})^H\Delta_0(u_1\otimes u_2\otimes u_3)\big|}
 {\big|(y_1^{(q)}\otimes y_2^{(q)}\otimes y_3^{(q)})^H\Delta_0(x_1^{(q)}\otimes x_2^{(q)}\otimes x_3^{(q)})\big|}< \xi
\end{equation}
for a given $\xi<1$, for instance $\xi=10^{-1}$.
Among those Ritz values satisfying the criterion, we choose the one closest to the prescribed
target.\medskip

\noindent
\textit{\textbf{Correction equation and preconditioning.}}
When the 
target is a point $(\lambda_0,\mu_0,\eta_0)$, we solve the correction equation in
line~\ref{lin1:correq} approximately by a Krylov subspace method, e.g., by GMRES. An important
feature of the Jacobi--Davidson method is the preconditioning applied to the correction
equation.
A good choice for a preconditioner is the inverse of $A_j-\lambda_0 B_j-\mu_0 C_j-\eta_0 D_j$.
Since this matrix has size $n_j\times n_j$, where $n_j$ is usually small compared to $n_1n_2n_3$,
this is a cheap operation.

If the
target is the plane $\eta=0$, then $\lambda_0$ and $\mu_0$ are not defined and we cannot use the preconditioning discussed above.
In this case, we often get good results if we solve the correction
equation exactly.
This is usually feasible, as in many applications $n_j$ is not large. We employ the expression
\[
	v_j=-u_j+(u_j^Hz_j)^{-1} \, z_j
\]
for the exact solution of the correction equation \eqref{eq:correq},
where $z_j := (A_j-\sigma B_j-\tau C_j-\psi D_j)^{-1}u_j$ for $j=1,2,3$; see \cite{SVDV} for the details. \medskip

\noindent
\textit{\textbf{Restarts.}} To keep the computation efficient, we restart Algorithm~\ref{alg:jd} in line~\ref{lin1:restart} when
the subspace becomes too large. As for the choice of
the new subspace of dimension $\ell$, we employ
\[U_j^{(k+1)}={\sf rgs}(u_j^{(k)}+v_j^{(k)},\ldots,u_j^{(k-\ell+1)}+v_j^{(k-\ell+1)}),\] where
$u_1^{(q)}\otimes u_2^{(q)}\otimes u_3^{(q)}$
is the Ritz vector 
and $v_j^{(q)}$, $j=1,2,3$, is the solution (exact or approximate) of the corresponding correction
equation \eqref{eq:correq} at iteration $q$. In this way, we build the new search space from the last $\ell$
eigenvector approximations. \medskip

\noindent
\textit{\textbf{Tensor Rayleigh Quotient Iteration.}}
The method performs better if we use Jacobi--Davidson up to a point when the residual of a Ritz pair is reasonably small,
i.e., smaller than $\delta$ in line~\ref{lin1:delta}, but still not smaller than $\varepsilon$
required for a convergence in line~\ref{lin1:convergence}. Whenever we find such a Ritz pair,
we refine it with the Tensor Rayleigh Quotient Iteration (TRQI),
which is a generalization of the standard Rayleigh quotient iteration and was also applied to a 2-parameter eigenvalue
problem in \cite{BorRD}.

Next we provide a brief description of the TRQI. An eigenpair of the 3-parameter eigenvalue problem is
a zero of the function
\[
 F(x,y,z,\lambda,\mu,\eta)
 =
 \left[
 \begin{matrix}(A_1-\lambda B_1-\mu C_1-\eta D_1)\,x\\[0.2em]
 (A_2-\lambda B_2-\mu C_2-\eta D_2)\,y\\[0.2em]
 (A_3-\lambda B_3-\mu C_3-\eta D_3)\,z\\[0.2em]
 u^Hx-1\\[0.2em]
 v^Hy-1\\[0.2em]
 w^Hz-1
 \end{matrix}
 \right],
\]
where constant vectors $u,v,w$, not orthogonal to $x,y,z$, respectively,
are used for normalization.
If $(x_k,y_k,z_k,\lambda_k,\mu_k,\eta_k)$ is an approximation
for a zero of $F$, then we may use Newton's method to obtain a new approximation
$(x_k+\Delta x_k,y_k+\Delta y_k,z_k+\Delta z_k,\lambda_k+\Delta\lambda_k,\mu_k+\Delta\mu_k,\eta_k+\Delta\eta_k)$.
In the TRQI, we start with an eigenvector approximation $x_k\otimes y_k\otimes z_k$, where $\|x_k\|=\|y_k\|=\|z_k\|=1$.
As an approximation
$(\lambda_k,\mu_k,\eta_k)$ for the corresponding eigenvalue, we use the tensor Rayleigh quotient
\begin{align*}
 \lambda_k &= \frac{(x_k\otimes y_k\otimes z_k)^H\Delta_1 (x_k\otimes y_k\otimes z_k)}
   {(x_k\otimes y_k\otimes z_k)^H\Delta_0 (x_k\otimes y_k\otimes z_k)},\\
 \mu_k &= \frac{(x_k\otimes y_k\otimes z_k)^H\Delta_2 (x_k\otimes y_k\otimes z_k)}
   {(x_k\otimes y_k\otimes z_k)^H\Delta_0 (x_k\otimes y_k\otimes z_k)},\\
 \eta_k &= \frac{(x_k\otimes y_k\otimes z_k)^H\Delta_3 (x_k\otimes y_k\otimes z_k)}
   {(x_k\otimes y_k\otimes z_k)^H\Delta_0 (x_k\otimes y_k\otimes z_k)},
\end{align*}
and set $x_{k+1},y_{k+1},z_{k+1}$ equal to the vectors $x_k+\Delta x_k,y_k+\Delta y_k,z_k+\Delta z_k$
that we get from one step of Newton's method
with an initial approximation $(x_k,y_k,z_k,\lambda_k,\mu_k,\eta_k)$. In this Newton step, we set $u=x_k$, $v=y_k$, and $w=z_k$.

Note that when none of $n_1,n_2,n_3$ is large, one step of the TRQI might be less expensive
than one iteration of the Jacobi--Davidson method and it is more efficient to switch to the TRQI to extract
the eigenpair once the Jacobi--Davidson method gets close enough.
The choice of the parameter $\delta$ requires care.
If we set $\delta$ too large, then the TRQI refinement is applied to poor
candidates, and the TRQI might converge to an eigenvalue that is not close to the target or an eigenvalue
that is already extracted. On the other hand, if $\delta$ is too small, then the condition in line~\ref{lin1:delta}
might never be fulfilled, and the method might not return any eigenvalues.
\medskip

\noindent
\textit{\textbf{Harmonic Ritz values.}} Last but not least, let us note
that although it is straightforward to generalize harmonic Ritz values from \cite{HP08} to 3-parameter eigenvalue problems,
we omit this ingredient in the description of the algorithm for simplicity.
We do not use harmonic Ritz values in the numerical experiments with
the Jacobi--Davidson method in Section~\ref{sec:NumResults},
but the use of harmonic Ritz values is an option in the implementation of Algorithm~\ref{alg:jd} in \MatlabPackage~\cite{BorMC1}.

\subsection{Use of a Krylov subspace method with full size tensor vectors}\label{subs:fullKrylov}
To find eigenvalues with the smallest $|\eta|$, we can also consider methods that
operate on the generalized eigenvalue problem $\Delta_3 z=\eta \Delta_0 z$. We present some alternatives
in this subsection and in the succeeding two subsections.

We consider a Krylov subspace method for \eqref{eq:3pdelta}, which means that in each step we have to solve
a linear system
	\begin{equation}\label{eq:sys3p}
		\Delta_3 w=\Delta_0 v
	 \end{equation}
for the unknown $w$ efficiently. While we can exploit the connection of such
linear systems to Sylvester equations in the 2-parameter
case, it does not seem straightforward to extend the Sylvester equation approach to the 3-parameter setting. Consequently,
it remains an open problem how to solve \eqref{eq:sys3p} with a complexity below ${\cal O}(n_1^3 n_2^3 n_3^3)$.

More specifically, by introducing the vectorizations $v={\rm vec}(\mathcal V)$ and $w={\rm vec}(\mathcal W)$, where
$\mathcal V,\mathcal W\in\CC^{n_1\times n_2\times n_3}$ are three dimensional tensors, it is
possible to express \eqref{eq:sys3p} as
\begin{equation}\label{eq:3pleft}\mathcal W\times_1 B_1\times_2 C_2\times_3 A_3\ +\
\mathcal W\times_1 C_1\times_2 A_2\times_3 B_3\ +\
\cdots\ -\ \mathcal W\times_1 A_1\times_2 C_2\times_3 B_3 =\mathcal{M},
\end{equation}
where the right hand side is
$
\mathcal{M}=\mathcal V\times_1 B_1\times_2 C_2\times_3 D_3\ +\ \mathcal V\times_1 C_1\times_2 D_2\times_3 B_3\ +\
\cdots\ -\ \mathcal V\times_1 D_1\times_2 C_2\times_3 B_3
$,
and $\times_j$ denotes the $j$-node product for $j=1,2,3$.
Equation \eqref{eq:3pleft} resembles a Sylvester equation in three
dimensions, but has too many terms. Namely, in three
dimensions the Sylvester equation has the form
\begin{equation}\label{eq:sylv3d}
 \mathcal{X}\times_1 A \ +\ \mathcal{X}\times_2 B \ +\ \mathcal{X}\times_3 C = \mathcal{Y}.
\end{equation}
Using
Schur decompositions for matrices
$A,B$, and $C$,
one can solve
\eqref{eq:sylv3d} efficiently by a generalization of the Bartels--Stewart algorithm;
see \cite{LiSylv} for details. Unfortunately, in our setting, we have six nonzero terms in \eqref{eq:3pleft},
and it does not seem possible to write this equation in the form \eqref{eq:sylv3d}.

\subsection{Subspace Iteration}\label{sec:3SubIter_v1}
If $n_1 n_2 n_3$ is too large for the approach in the previous subsection,
then we can apply subspace iteration to \eqref{eq:3pdelta}
in a way similar to its counterpart for the 2-parameter case, using low-rank approximations to make the computation feasible.
\emph{The exact subspace iteration} with full vectors operates as follows. We start with a matrix $Z_0\in\CC^{n_1n_2n_3\times p}$ with orthonormal columns.
In each step, for a given $Z_k$, we solve the linear system $\Delta_3 W_k=\Delta_0 Z_k$ for $W_k$, and then set $Z_{k+1}$ equal to the $Q$
factor in the QR decomposition of $W_k$. Typically, the columns of $Z_k$ converge to an orthonormal basis for the dominant invariant subspace of
$\Delta_3^{-1}\Delta_0$, and $Z_k^T\Delta_3^{-1}\Delta_0Z_k$ 
converges to an upper triangular matrix with $p$ dominant eigenvalues $\eta$ of \ref{eq:3pdelta}
on the diagonal.

As the full columns of $Z_k$ are too large, we use low-rank approximations.
We call this variant \emph{inexact subspace iteration}.
Specifically, we suppose that all columns of $Z_k\in\CC^{n_1n_2n_3\times p}$ lie in a subspace spanned by
$U_1^{(k)}\otimes U_2^{(k)}\otimes U_3^{(k)}$ for $U_j^{(k)}\in\CC^{n_j\times \ell}$ for $j = 1,2,3$.
The columns of $Z_k\in\CC^{n_1n_2n_3\times p}$ are represented in the Tucker format
\begin{equation}\label{eq:tucker3}
 z_i^{(k)}={\rm vec}({\mathcal D}_i^{(k)}\times_1 U_1^{(k)}\times_2 U_2^{(k)}\times_3 U_3^{(k)}),
\end{equation}
where ${\mathcal D}_i^{(k)}$ is an $\ell\times\ell\times\ell$ core tensor for $i=1,\ldots,p$,
where $p\le\ell$. Each iteration proceeds as follows.
\begin{enumerate}
 \item[1)] Solve the linear system $\Delta_3 w_i^{(k)}=\Delta_0 z_i^{(k)}$ approximately for $i=1,\ldots,p$
 (see below for details), and orthonormalize the solution vectors $w_1^{(k)},\ldots,w_p^{(k)}$.
 \item[2)] Replace the orthonormalized solutions with their low-rank approximations, which leads to
 $z_1^{(k+1)},\ldots,z_p^{(k+1)}$ forming the columns of $Z_{k+1}$ for the next step.
\end{enumerate}
In the second step, $z_i^{(k+1)} = {\rm vec}({\mathcal D}_i^{(k+1)} \times_1 U_{1}^{(k+1)} \times_2 U_{2}^{(k+1)} \times_3 U_{3}^{(k+1)})$
for some $U_{j}^{(k+1)} \in \CC^{n_j\times \ell}, \; j = 1,2,3$. We explain how to form $U_j^{(k+1)}$ in Algorithm~\ref{alg:uk1}
at the end of this subsection.

The main part of the  inexact subspace iteration 
is to solve the linear systems $\Delta_3 w_i^{(k)}=\Delta_0 z_i^{(k)}$ for $i=1,\ldots,p$ approximately
by using low-rank approximations. This 
is justified by the following argument. When $v$ in \eqref{eq:sys3p}
is an eigenvector of \eqref{eq:3pdelta}, which implies $v=v_1\otimes v_2\otimes v_3$ is a decomposable tensor,
then the right-hand side of \eqref{eq:sys3p} is a sum
\[
 B_1v_1\otimes C_2v_2\otimes D_3v_3+ C_1v_1\otimes D_2v_2\otimes B_3v_3 +\cdots - D_1v_1\otimes C_2v_2\otimes B_3v_3
\]
of six rank-one tensors.
In this case, the solution $w$ of \eqref{eq:sys3p} is also an eigenvector of \eqref{eq:3pdelta}, and
has rank one. As in the exact subspace iteration, the columns of $Z_k$ converge to
linear combinations of a small number of dominant eigenvectors and it is 
reasonable to use low-rank approximations
for the solutions of the linear systems $\Delta_3 w_i^{(k)}=\Delta_0 z_i^{(k)}$ for $i=1,\ldots,p$.

Although we cannot write \eqref{eq:3pleft} as a Sylvester equation in the 3-parameter setting, we can borrow some ideas from
the Krylov method for the 2-parameter case that is based on the
solutions of Sylvester equations by the low-rank approximation
approach due to Hu--Reichel. In particular, suppose that we are looking for a low-rank approximation of the solution of
\eqref{eq:sys3p}.
Let us assume that $A_1,A_2$, and $A_3$ are nonsingular.
Then \eqref{eq:sys3p} is equivalent to
\begin{equation}\label{eq:btilda}\nonumber
\widetilde \Delta_3 w \;\; := \;\; \left|
 \begin{array}{ccc}
 \widetilde B_1 & \widetilde C_1 & I \\
 \widetilde B_2 & \widetilde C_2 & I \\
 \widetilde B_3 & \widetilde C_3 & I
 \end{array}
 \right|_\otimes w \;\; = \;\;
 \left|
 \begin{array}{ccc}
 \widetilde B_1 & \widetilde C_1 & \widetilde D_1 \\
 \widetilde B_2 & \widetilde C_2 & \widetilde D_2 \\
 \widetilde B_3 & \widetilde C_3 & \widetilde D_3
 \end{array}
 \right|_{\otimes}v \;\; =: \;\; \widetilde \Delta_0 v,
\end{equation}
where $\widetilde\Delta_3=(A_1\otimes A_2\otimes A_3)^{-1}\Delta_3$,
$\widetilde\Delta_0=(A_1\otimes A_2\otimes A_3)^{-1}\Delta_0$, $\widetilde{B}_i=A_i^{-1}B_i$, $\widetilde{C}_i=A_i^{-1}C_i$,
$\widetilde{D}_i=A_i^{-1}D_i$. Observe that
for $v$ of the form $v = {\rm vec }({\mathcal D}\times_1 U_1^{(k)} \times_2 U_2^{(k)} \times_3 U_3^{(k)} )$
for some ${\mathcal D}$, the vector $\widetilde \Delta_0 v$
lies in the subspace spanned by $F_1\otimes F_2\otimes F_3$,
where $F_j={\rm span} \{ \widetilde B_j U_j^{(k)},\widetilde C_j U_j^{(k)}, \widetilde D_j U_j^{(k)} \}$
for $j=1,2,3$. Our low-rank approach employs the generalized Krylov subspaces
\begin{equation}\label{eq:KrlyGNew}
 {\mathcal K}_r(\widetilde B_j,\widetilde C_j,F_j):={\rm span}\{
 M_0(\widetilde B_j,\widetilde C_j, F_j),\,
 M_1(\widetilde B_j,\widetilde C_j, F_j),\,
 \ldots,\,
 M_{r-1}(\widetilde B_j,\widetilde C_j, F_j)
 \}
\end{equation}
for a modest $r$,
where $M_0(\widetilde B_j,\widetilde C_j, F_j)=F_j$ and
$M_{i+1}(\widetilde B_j,\widetilde C_j, F_j) =
\big[\widetilde B_jM_i(\widetilde B_j,\widetilde C_j, F_j) \ \ \widetilde C_jM_i(\widetilde B_j,\widetilde C_j, F_j)\big]$
for $i>0$. This is a generalization of the Krylov subspaces used in the Hu--Reichel method; cf.~\cite{LiYe}.
An approximate solution of \eqref{eq:sys3p} is assumed to be of the form
\[
 w = {\rm vec}(\mathcal Y\times_1 Q_1\times_2 Q_2\times_3 Q_3)
 \in {\mathcal K}_r(\widetilde B_1,\widetilde C_1,F_1)\otimes
 {\mathcal K}_r(\widetilde B_2,\widetilde C_2,F_2)\otimes
 {\mathcal K}_r(\widetilde B_3,\widetilde C_3,F_3),
\]
where $Q_j$ is a matrix whose columns form an orthonormal basis for ${\mathcal K}_r(\widetilde B_j,\widetilde C_j,F_j)$, and
$\mathcal Y$ is the solution of the projected equation
\[
 \left|
 \begin{array}{ccc}
 Q_1^H\widetilde B_1Q_1 & Q_1^H\widetilde C_1Q_1 & I \\
 Q_2^H\widetilde B_2Q_2 & Q_2^H\widetilde C_2Q_2 & I \\
 Q_3^H\widetilde B_3Q_3 & Q_3^H\widetilde C_3Q_3 & I
 \end{array}
 \right|_\otimes
 {\rm vec}({\mathcal Y}) \;\; = \;\;
 \left|
 \begin{array}{ccc}
 Q_1^H\widetilde B_1Q_1 & Q_1^H\widetilde C_1Q_1 & Q_1^H\widetilde D_1Q_1 \\
 Q_2^H\widetilde B_2Q_2 & Q_2^H\widetilde C_2Q_2 & Q_2^H \widetilde D_2Q_2 \\
 Q_3^H\widetilde B_3Q_3 & Q_2^H\widetilde C_3Q_3 & Q_3^H\widetilde D_3 Q_3
 \end{array}
 \right|_{\otimes} (Q_1^H\otimes Q_2^H\otimes Q_3^H) v
\]
that satisfies the Galerkin condition that the residual is orthogonal to the subspace ${\rm span} \{ Q_1\otimes Q_2\otimes Q_3 \}$.
In the 2-parameter case, we can exploit the relation to the Sylvester equation to solve the projected equation efficiently. 
As explained in the previous subsection, we are not aware of such a relation
in the 3-parameter setting.
Hence, we solve the projected systems directly. For this reason, the dimension of the subspace $Q_1\otimes Q_2\otimes Q_3$
cannot grow too large.

The above procedure yields vectors $w_i^{(k)}={\rm vec}({\mathcal Y}_i \times_1 Q_1\times_2 Q_2\times_3 Q_3)$ for $i=1,\ldots,p$, which are
orthonormalized into $\widetilde w_i^{(k)}={\rm vec}(\widetilde {\mathcal Y}_i \times_1 Q_1\times_2 Q_2\times_3 Q_3)$ for $i=1,\ldots,p$
by the Gram--Schmidt procedure. We remark that the orthonormalization affects only the core tensors while the subspace bases
$Q_1$, $Q_2$, $Q_3$ do not change. After orthonormalization, we approximate $\widetilde w_1^{(k)},\ldots,\widetilde w_p^{(k)}$
by their orthogonal projections onto a low-dimensional subspace ${\rm span} \{ U_1^{(k+1)}\otimes U_2^{(k+1)}\otimes U_3^{(k+1)} \}$
for some $U_{j}^{(k+1)} \in \CC^{n_j\times \ell}$ such that ${\rm span} \{ U_j^{(k+1)} \} \; \subset \; {\rm span}\{ Q_j \}$.

Finally, we discuss
a feasible approach to construct a suitable $\ell\times \ell \times \ell$ dimensional subspace
${\rm span} \{ U_1^{(k+1)} \otimes U_2^{(k+1)} \otimes U_3^{(k+1)} \}$ of ${\rm span}\{Q_1\otimes Q_2\otimes Q_3\}$.
As we apply a subspace iteration, we
expect that, near convergence, $\widetilde w_1^{(k)}$ is close to the dominant eigenvector,
which is a decomposable tensor. Furthermore, $\widetilde w_2^{(k)}$ should be close to a linear combination of the dominant two
eigenvectors, and so on. Thus, we construct the $\ell\times \ell\times \ell$ dimensional subspace
${\rm span} \{ U_1^{(k+1)}\otimes U_2^{(k+1)}\otimes U_3^{(k+1)} \}$ by considering $\widetilde w_1^{(k)}$ first. We determine
a subspace ${\rm span}\{ V_1\otimes V_2\otimes V_3\}$ that contains a good low-rank
approximation of $\widetilde w_1^{(k)}$, where $V_j$ is an $n_j\times m_j$ matrix with orthonormal columns and $m_j \le \ell$ for $j=1,2,3$.
While the best low-rank approximation is well-defined and easy to compute in the 2-parameter case,
this is more complicated in the 3-parameter setting, where the available tools
are the multilinear singular value decomposition or a low multilinear rank approximation, see, e.g., \cite{Bader}. Once we obtain a low-rank
approximation for $\widetilde w_1^{(k)}$, we take $V_j$ as the starting column block of $U_j^{(k+1)}$. Then we
project the next vector $\widetilde w_2^{(k)}$ onto the orthogonal complement of
${\rm span} \{ V_1\otimes V_2\otimes V_3 \}$, and find a new column block for
$U_j^{(k+1)}$
from a low-rank approximation of the projected vector. We continue this way 
until
we collect enough columns for $U_j^{(k+1)}$ for $j=1,2,3$.
This construction is 
described
in Algorithm~\ref{alg:uk1}. Finally, it is worth remarking that for small $k$, when the subspace is far from
 an invariant one,
 we can expect to get all columns of $U_j^{(k+1)}$ for $j=1,2,3$
just from a low-rank approximation of $\widetilde w_1^{(k)}$, while at later iterations, after $\widetilde w_1^{(k)}$
has already converged to a dominant eigenvector, we obtain only the first column of $U_j^{(k+1)}$ from $\widetilde w_1^{(k)}$,
and the remaining ones 
from $\widetilde w_2^{(k)},\ldots,\widetilde w_p^{(k)}$.

\begin{algorithm}[ht]
\caption{\emph{Computation of matrices $U_1^{(k+1)},U_2^{(k+1)},U_3^{(k+1)}$ with orthonormal columns forming a basis
for low-rank approximations of $w_1^{(k)},\ldots,w_p^{(k)}$, which are the vectors generated by the inexact subspace iteration
at step $k$. \label{alg:uk1}}}
\begin{algorithmic}[1]
\STATE $U_1^{(k+1)}=[]$, $U_2^{(k+1)}=[]$, $U_3^{(k+1)}=[]$, $m=0$
\FOR {$q=1, \ldots,p$ and while $m<\ell$}
\STATE $z= \left( I \: - \: [U_1^{(k+1)}\otimes U_2^{(k+1)}\otimes U_3^{(k+1)}] [U_1^{(k+1)}\otimes U_2^{(k+1)}\otimes U_3^{(k+1)}]^H \right)w_q^{(k)}$
\STATE Find matrices $V_1,V_2,V_3$ with $s\le \ell-m$ orthonormal columns that are used for a low rank approximation of $z$
in ${\rm span}\{V_1\otimes V_2\otimes V_3\}$.
\STATE $U_j^{(k+1)}= \big[ U_j^{(k+1)} \; V_j \big]$ for $j = 1,2,3$.
\STATE $m=m+s$.
\ENDFOR
\end{algorithmic}
\end{algorithm}

\subsection{Subspace iteration with Arnoldi expansion}\label{sec:3SubIter}
The inexact subspace
iteration for the 3-para\-me\-ter eigenvalue problem presented
in the previous subsection is inspired from the ideas in \cite{KarlBor}
for the 2-parameter case. Here, we further simplify that approach by avoiding
the explicit use of low-rank approximations, giving rise to a method that is easier to implement.
We will add some new features that are not present in the 2-parameter version in \cite{KarlBor},
which improve the efficiency of the approach substantially in the 3-parameter case.
Some of the new features, for instance the selection criterion
from Section~\ref{subsec:JD}, can be adopted in the 2-parameter version in a straightforward way.

In the inexact subspace iteration of the previous subsection, the approximate solutions of the linear systems
at step $k$ are assumed to lie in ${\rm span} \{ Q_1\otimes Q_2\otimes Q_3 \}$, where the columns of $Q_j$ form
an orthonormal basis for the generalized Krylov subspace ${\mathcal K}_r(\widetilde B_j,\widetilde C_j,F_j)$ defined in
\eqref{eq:KrlyGNew} for $j=1,2,3$. These spaces contain many approximations for the eigenvectors that we can use
to form the next subspace ${\rm span} \{ U_1^{(k+1)}\otimes U_2^{(k+1)}\otimes U_3^{(k+1)} \}$.
Here,
 we form the
subspace from $\ell$ Ritz vectors of 
the $\ell$ Ritz values with the smallest
$|\psi|$ of the projected 3-parameter eigenvalue problem
\begin{align}
Q_1^HA_1Q_1 \, s_1 &= \sigma \, Q_1^HB_1Q_1 \, s_1+ \tau \, Q_1^HC_1Q_1 \, s_1 + \psi \, Q_1^HD_1Q_1 \, s_1 \nonumber \\
Q_2^HA_2Q_2 \, s_2 &= \sigma \, Q_2^HB_2Q_2 \, s_2+ \tau \, Q_2^HC_2Q_2 \, s_2 + \psi \, Q_2^HD_2Q_2 \, s_2 \label{eq:proj3p} \\
Q_3^HA_3Q_3 \, s_3 &= \sigma \, Q_3^HB_3Q_3 \, s_3+ \tau \, Q_3^HC_3Q_3 \, s_3 + \psi \, Q_3^HD_3Q_3 \, s_3. \nonumber
\end{align}
As each Ritz vector is decomposable, we form $U_j^{(k+1)}$ such that its columns form an orthonormal basis for the subspace
spanned by the $j$-nodes of the selected Ritz vectors.
Note that this approach is close to the methods based on tensor decompositions such as those in \cite{krsv16} and \cite{krto11}.
The main difference is that our approach only uses the factor matrices $Q_1,Q_2,Q_3$ of a Tucker decomposition, i.e.,
the core tensor is not used.
A formal description of the approach is given in Algorithm~\ref{alg3} and
some details are discussed below.\medskip

\begin{algorithm}[ht]
\caption{\emph{Subspace iteration with Arnoldi expansion and restarts based on selected Ritz vectors for the generalized
eigenvalue problem \eqref{eq:sys3p} associated with the 3-parameter eigenvalue problem.} \\
In the algorithm, $\ell$ denotes the size of the subspace after a restart, $r$ is the number of block Arnoldi steps,
$\varepsilon$ is used in the convergence criterion for an eigenvalue, and $\delta > \varepsilon$ controls when
a Ritz pair is a candidate for the TRQI refinement.
\label{alg3}}
\begin{algorithmic}[1]
\STATE Choose initial matrices $U_j^{(0)}\in\CC^{n_j\times \ell}$ with orthonormal columns for $j=1,2,3$.
\FOR {$k=0, 1, \ldots$}
\FOR {$j=1,2,3$}
\STATE $F_j=[A_j^{-1}B_jU_j^{(k)} \ \ A_j^{-1}C_jU_j^{(k)} \ \ A_j^{-1}D_jU_j^{(k)}]$\label{lin2:f}
\STATE Form $Q_j$ whose columns are orthonormal basis
for ${\cal K}_r(A_j^{-1}B_j,A_j^{-1}C_j,F_j)$ using a block Arnoldi algorithm with SVD filtering;
see Algorithm~\ref{alg:blockAr}.\label{lin2:arnoldi}
\ENDFOR
\IF{the size of $Q_1\otimes Q_2\otimes Q_3$ is too large}
\STATE Shrink matrices $Q_1$, $Q_2$, $Q_3$ by the same factor by removing the appropriate number of the last columns.\label{lin2:shrink}
\ENDIF
\STATE Compute $m$ Ritz values $(\sigma_i,\tau_i,\psi_i)$ and 
Ritz vectors\label{lin2:extract}
$z_1^{(i)}\otimes z_2^{(i)}\otimes z_3^{(i)}:=Q_1s_1^{(i)} \otimes Q_2s_2^{(i)}\otimes Q_3s_3^{(i)}$ for $i=1,\ldots,m$
with the smallest values of $|\psi|$ from the projected 3-parameter eigenvalue problem \eqref{eq:proj3p}.
\STATE Refine Ritz pairs $\big((\sigma_i,\tau_i,\psi_i),z_1^{(i)}\otimes z_2^{(i)}\otimes z_3^{(i)}\big)$\label{lin2:refine}
for $i=1,\ldots,m$
by $s\ge 0$ steps of the TRQI.
\FOR {$i=1,\ldots,m$}
\IF {the Ritz pair $\big((\sigma_i,\tau_i,\psi_i),\,z_1^{(i)}\otimes z_2^{(i)}\otimes z_3^{(i)}\big)$ satisfies the selection criterion}\label{lin2:usesel}
\STATE Compute the residual $r_{ij}=(A_j-\sigma_i B_j-\tau_i C_j - \psi_i D_j)\,z_j^{(i)}$ for $j=1,2,3$.
\IF {$(\|r_{i1}\|^2+\|r_{i2}\|^2+\|r_{i3}\|^2)^{1/2}\le \delta$}\label{lin2:delta}
\STATE Further refine the Ritz pair with $t\ge 0$ steps of the TRQI and update the residuals.\label{lin2:furtrefine}
\IF{refined pair satisfies the selection criterion \textbf{and} $(\|r_{i1}\|^2+\|r_{i2}\|^2+\|r_{i3}\|^2)^{1/2}\le \varepsilon$}\label{lin2:convergence0}
\STATE Extract the eigenpair and compute the corresponding left eigenvector
\label{lin2:convergence}
\ENDIF
\ENDIF
\ENDIF
\ENDFOR
\STATE Let $p_1,\ldots,p_\ell$ be the indices of the first $\ell$ Ritz pairs that satisfied the selection criterion, but did not lead
to an eigenpair.
\STATE Form $U_j^{(k+1)}$
whose columns make
an orthonormal basis for
${\rm span} \{ z_j^{(p_1)},\ldots,z_j^{(p_\ell)} \}$ for $j=1,2,3$.\label{lin2:restart}
\ENDFOR
\end{algorithmic}
\end{algorithm}

\noindent
\textit{\textbf{Block Arnoldi Algorithm with SVD Filtering.}} The block Arnoldi algorithm in line \ref{lin2:arnoldi}
employed together with an SVD filtering
is presented in Algorithm~\ref{alg:blockAr}. In the 3-parameter setting, we are quite limited
in the maximum search space. In particular, if the size of the subspace ${\rm span} \{ Q_1\otimes Q_2\otimes Q_3 \}$
is too large, then we cannot solve the projected problem in line \ref{lin2:extract}.
 Hence, we use the SVD filtering
and the relative cutoff parameter $\zeta \ge 0$
to prevent on the one hand the search space to grow
too much, and on the other hand 
to keep all the significant directions in the subspace. In our experiments,
$\zeta = 10^{-5}$
gives good results in practice.\medskip

\noindent
\textit{\textbf{Selection Criterion.}} In line \ref{lin2:usesel} of Algorithm~\ref{alg3}, we use the same selection criterion
as in the Jacobi--Davidson method,
defined by \eqref{eq:selcrit}.
As we need the left eigenvectors corresponding to the eigenvalues that are already extracted
to check this criterion, we compute a left eigenvector in line~\ref{lin2:convergence} for each new eigenvalue that we
find. If a Ritz pair satisfies the selection criterion, it can still happen that the TRQI refinement converges to one of the
eigenvalues that is already extracted. Therefore, we test the selection criterion in line \ref{lin2:convergence0} once again
to make sure that an eigenvalue is not repeated.\medskip

\noindent
\textit{\textbf{TRQI Refinement.}} The convergence can be drastically improved if we refine all Ritz pairs
with a small number of TRQI steps in line~\ref{lin2:refine} of Algorithm~\ref{alg3}. This improves the directions that we use for a restart
in line~\ref{lin2:restart}, additionally it yields more candidates that satisfy the criterion in line~\ref{lin2:delta}. However,
we should not use too many refinement steps because even when the TRQI is applied to a poor approximation,
it can still converge to an eigenpair. In most cases, such a converged eigenpair is not close to the prescribed
target (e.g., it does not have a small $|\eta|$), or is an eigenpair that is already extracted.

If, after this initial TRQI refinement, the selection criterion is satisfied by a Ritz pair and the norm of
the corresponding residual is below $\delta$ in line~\ref{lin2:delta}, then
the TRQI refinement is applied once again to the candidate Ritz pair.
As in Algorithm~\ref{alg:jd}, the parameter
$\delta$ should be chosen with care.
Since we use only a few
steps of block Arnoldi to form our search space, we cannot expect it to contain very good approximations of the eigenvectors.
Hence, we perform the second stage of the TRQI on approximations with residuals that are reasonably small to
overcome their inaccuracy due to the crudeness of the subspaces.

\begin{algorithm}[ht]
\caption{\emph{Block Arnoldi expansion with an SVD filtering to form an orthonormal basis for ${\cal K}_r(B,C,F)$.} \\
In the algorithm, $\zeta\ge 0$ denotes the relative cutoff parameter for the singular values.
\label{alg:blockAr}}
\begin{algorithmic}[1]
\STATE Compute the singular value decomposition $F=U\Sigma V^T$.
\STATE Select $W=[u_1\ \cdots\ u_j]$, where $j$ is such that $\sigma_j\ge \zeta\sigma_1 >\sigma_{j+1}$, or \\
		$j$ is the number of columns of $F$.
\STATE $Q = W$
 \FOR {$k=1, \ldots,r$}
 \STATE $G=(I-QQ^H)[BW \ \ CW]$
 \STATE Compute the singular value decomposition $G=U\Sigma V^T$.
 \STATE Select $W=[u_1\ \cdots\ u_j]$, where $j$ is such that $\sigma_j\ge \zeta \sigma_1 >\sigma_{j+1}$, or \\
 		$j$ is the number of columns of $G$.
 \STATE $Q = [Q \ \ W]$
\ENDFOR
\end{algorithmic}
\end{algorithm}

\section{Numerical results}\label{sec:NumResults}
Algorithm~\ref{alg:jd} and Algorithm~\ref{alg3} are both implemented in Matlab package \MatlabPackage~\cite{BorMC1}. In this section,
we conduct numerical experiments with these implementations on several 3-parameter eigenvalue problems;
all 
of these examples are available in \MatlabPackage. The results have been obtained
using Matlab R2012b on a PC having 16GB RAM and an i5-4670 3.4 GHz CPU.

\subsection{Ellipsoidal wave equation}\label{sec:result_ellip_wave}
The first two numerical experiments are performed on the ellipsoidal wave equation described in Section~\ref{sec:motive_ellip_wave}
with the particular choices $x_0 = 1, y_0 = 1.5$, and $z_0 = 2$ for the radii of the semi-axes of the ellipsoid
and $\rho=\sigma=\tau=0$ for the configuration.
This problem was solved numerically using matrices of size $25\times 25$ and the approach from
Section~\ref{subs:fullKrylov} in \cite{Calin3}. Using Algorithms~\ref{alg:jd} and \ref{alg3}, we can work with
much larger matrices corresponding to finer discretizations, and obtain more accurate results for
the low eigenfrequencies.

We discretize \eqref{eq:main_ODE_sys} using the Chebyshev collocation on 300 points.
We know that all eigenvalues $(\lambda,\mu,\eta)$ of \eqref{eq:main_ODE_sys}
are real and such that $\eta>0$, see, e.g., \cite{Levitina}. As we are
interested in eigenvalues with $\eta$ closest to the target $\eta_{\rm tar}\ge 0$,
we apply the substitution
$(\widetilde{\lambda}, \widetilde{\mu},\widetilde{\eta})
 	=
 (\lambda+5,\mu,\eta - \eta_{\rm tar})$
 and search for eigenvalues close to $\widetilde \eta=0$ of the transformed problem, with
the coefficient matrices $\widetilde A_j=A_j+5B_j-\eta_{\rm tar} D_j$, $\widetilde B_j=B_j$,
$\widetilde C_j=C_j$, and $\widetilde D_j=D_j$ for $j=1,2,3$. We use the shift $5$
(where $5$ is more or less randomly chosen) to make $\widetilde A_j$ nonsingular in the case $\eta_{\rm tar}=0$; it
changes the $\lambda$ components of the eigenvalues,
but does not affect our search which is based on a prescribed target on the $\eta$ components of the eigenvalues.

Before applying the numerical methods we 
multiply the $j$th equation by $\widetilde A_j^{-1}$ for $j = 1,2,3$, after ensuring that
$\widetilde A_j$ is nonsingular.
This 
is equivalent to considering the  generalized eigenvalue problem
\vskip -2ex
\begin{equation}\label{eq:invertDelta}
(\widetilde A_1 \otimes \widetilde A_2\otimes \widetilde A_3)^{-1} \widetilde \Delta_3 z =
 \eta \, (\widetilde A_1 \otimes \widetilde A_2\otimes \widetilde A_3)^{-1} \widetilde \Delta_0 z
\end{equation}
instead of
$\widetilde \Delta_3 z =\eta \widetilde \Delta_0 z$. We do this because the Chebyshev collocation returns matrices
such that $\|\widetilde A_j\|\gg\|\widetilde B_j\|,\|\widetilde C_j\|,\|\widetilde D_j\|$
and $\widetilde A_j$ is ill-conditioned
for $j=1,2,3$, where $\| \cdot \|$ denotes the matrix 2-norm. These facts in turn imply that
$\|\widetilde \Delta_3\|\gg \|\widetilde \Delta_0\|$ and $\widetilde{\Delta}_3$ is ill-conditioned.
We expect that Ritz values of \eqref{eq:invertDelta} are better approximations
for the eigenvalues with the smallest value of $|\eta|$.
\smallskip

\begin{example}[Jacobi--Davidson on the Ellipsoidal Wave Equation]\label{ex:ell1}
{\rm
We apply Algorithm~\ref{alg:jd} where
we set the plane $\eta=0$ as the target and solve the correction equation exactly.
We restrict the subspace dimensions between 5 and 10; in particular we restart using the 
eigenvector approximations from the last five iterations. In line~\ref{lin1:extraction} the Ritz values
are arranged in increasing order according to their distances from the target. We consider a Ritz pair
as a candidate for an eigenpair if its residual is smaller than $\delta=10^{-1}$, and
if it satisfies the selection criterion \eqref{eq:selcrit} with $\xi_1=10^{-1}$. In this case,
we refine the Ritz pair with up to 4 steps of
the TRQI. After the refinement, if the residual drops below $\varepsilon=10^{-8}$ and if the
selection criterion with $\xi_2=10^{-4}$ is satisfied, then the refined pair is accepted as a new eigenpair.
We can extract more than one eigenvalue from the same subspace
(without executing the else part of the if statement, that is without executing lines \ref{lin1:correq}--\ref{lin1:restart}).

We have computed 80 eigenvalues for the following three cases:
\begin{enumerate}
\item[a)] $\eta_{\rm tar}=0$, the target eigenvalues are exterior;
\item[b)] $\eta_{\rm tar}=200$, the desired eigenvalues are close to the exterior ones, since there are only a
few hundred eigenvalues with their $\eta$ component satisfying $\eta<200$;
\item[c)] $\eta_{\rm tar}=1000$, the target eigenvalues are mildly interior and more
difficult to compute.
\end{enumerate}
The computational times for cases a), b) and c) are 7, 480 and 540 seconds for
15, 472 and 594 iterations, respectively.
Figure \ref{fig:elljd1} shows the values of $|\eta-\eta_{\rm tar}|$ of the computed eigenvalues in
the order of retrieval. In case a) the eigenvalues converge almost in the desired order.
In case b) the eigenvalues are not computed
in such a desirable order (i.e., the monotonicity of the distances of the $\eta$ components to the prescribed target
with respect to the order of the retrieval degrades slightly), but the method still extracts the eigenvalues close
to the prescribed target. In case c) the eigenvalues
are retrieved even in a less-structured order and we need to compute many eigenvalues to be sure that we get the
desired eigenvalues closest to the target.

\begin{figure}[ht]
\begin{center}
\includegraphics[scale=0.047]{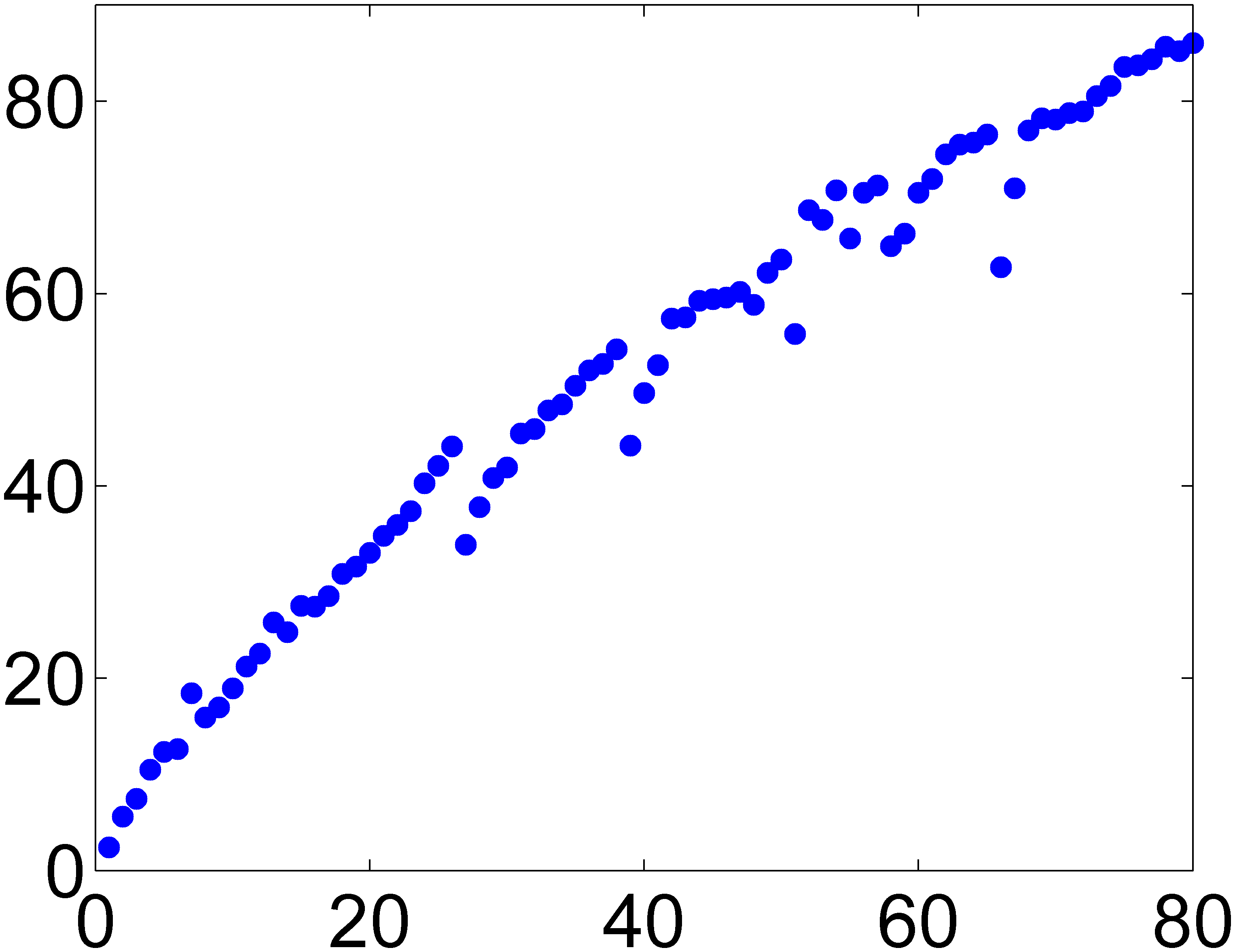} \
\includegraphics[scale=0.047]{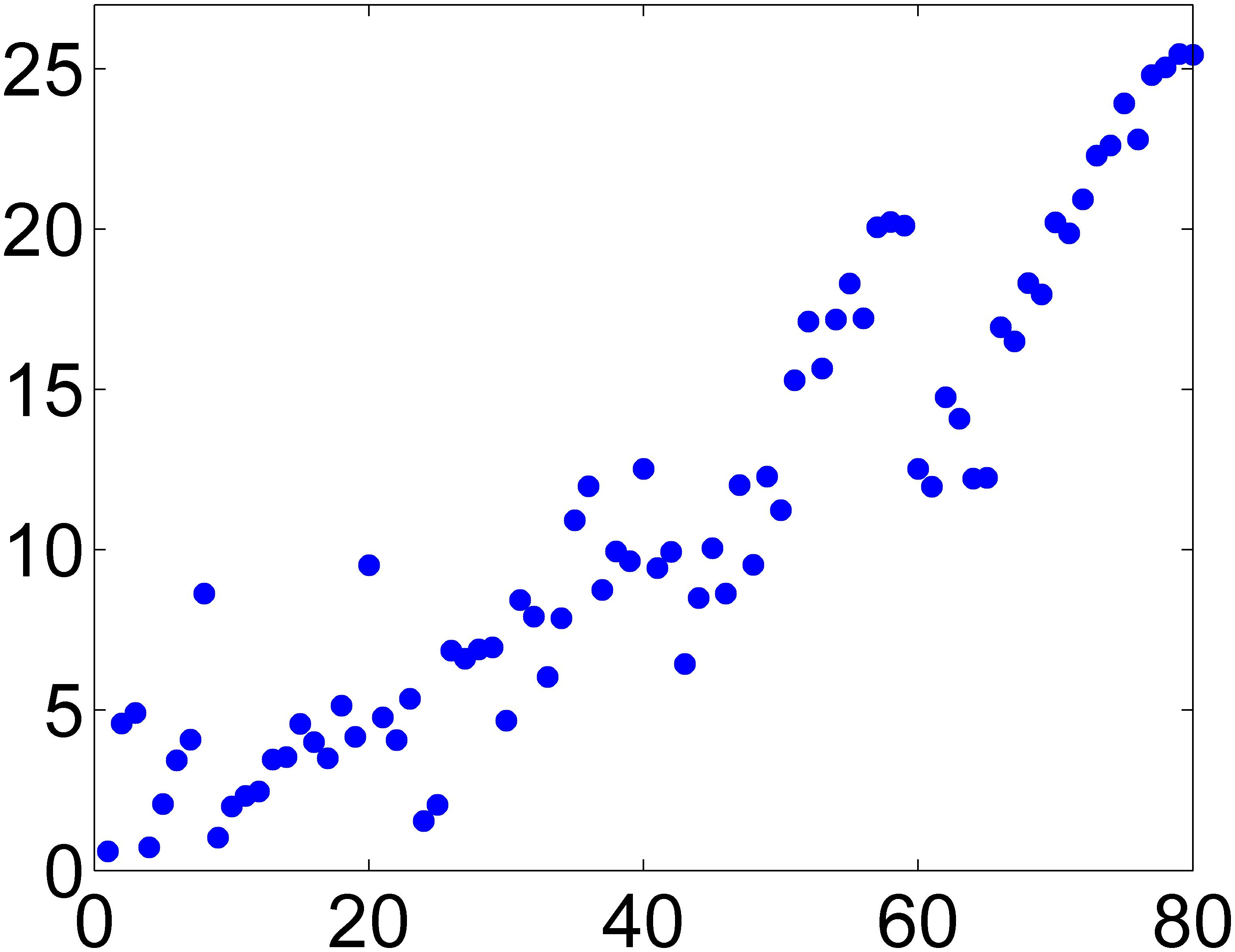} \
\includegraphics[scale=0.047]{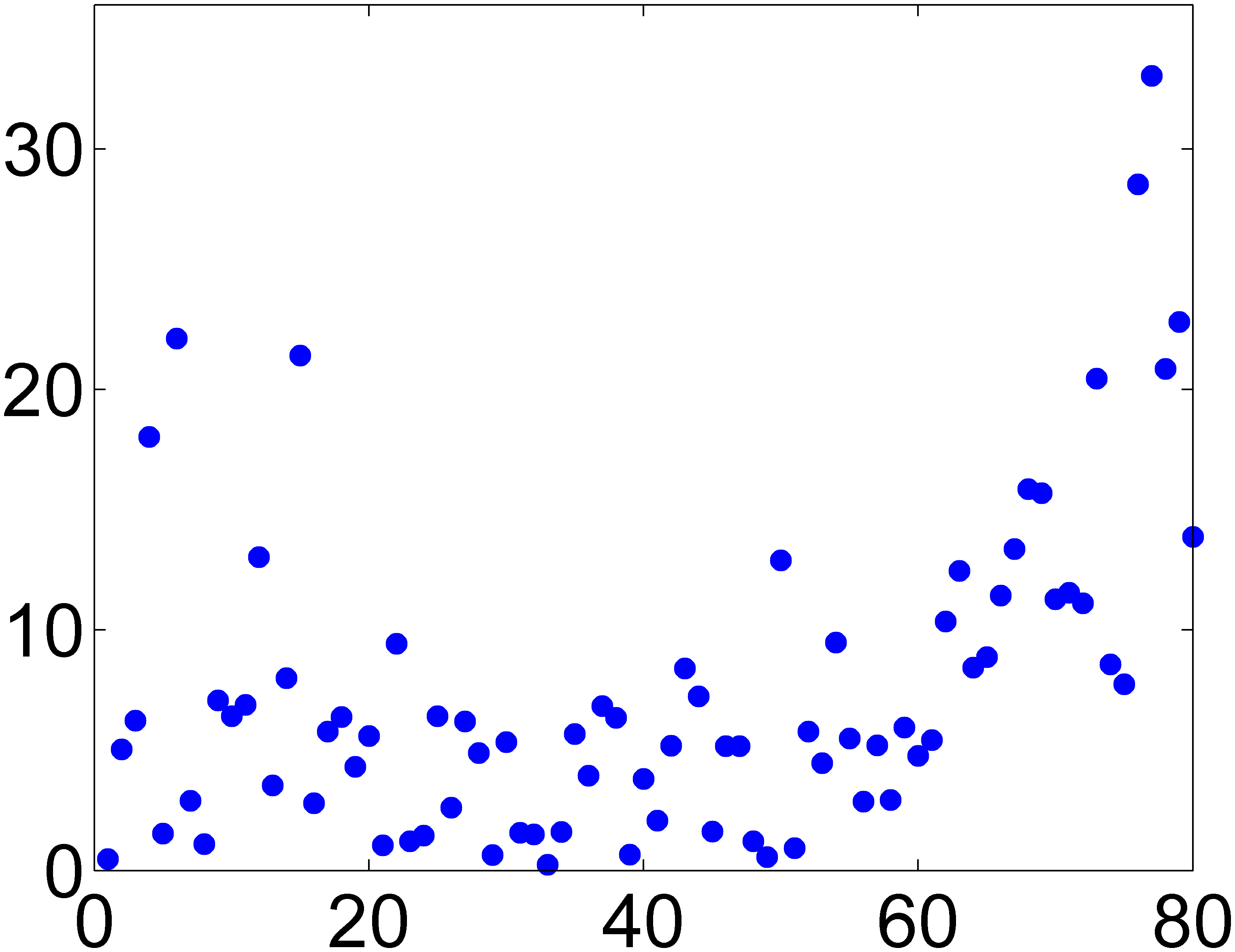}
\end{center}
\vspace{-0.7em}
\caption{Jacobi--Davidson method for the 3-parameter
eigenvalue problem in Section~\ref{sec:result_ellip_wave}.
The values $|\eta-\eta_{\rm tar}|$ (vertical axis) of the first 80 computed eigenvalues $(\lambda,\mu,\eta)$
are plotted with respect to the order of retrieval (horizontal axis) for the following cases:
a) $\eta_{\rm tar}=0$; (left), b) $\eta_{\rm tar}=200$ (middle); c) $\eta_{\rm tar}=1000$ (right).}
\label{fig:elljd1}
\vspace{-1em}
\end{figure}

We explored how many eigenvalues one needs to compute with the above settings
to get the first 40, 20 and 10 eigenvalues with their $\eta$ components closest to $\eta_{\rm tar}$
for cases a), b) and c), respectively. We decrease the number of targeted eigenvalues
for larger values of $\eta_{\rm tar}$, as interior eigenvalues are more difficult to compute.
In Table \ref{tab:ell_avg}, we report average results together with the best and worst run
of the algorithm over a set of 10 different random initial subspaces. To make sure that we
have all of the closest eigenvalues so that the comparisons are fair, we have computed
the eigenvalues a priori repeatedly several times.

\begin{table}[htb]
\caption{The Jacobi--Davidson method for the 3-parameter eigenvalue problem
in Section~\ref{sec:result_ellip_wave}.
Total number of eigenvalues that had to be computed, number of subspace updates that had to be performed
and computational times in order to retrieve the targeted number of eigenvalues $(\lambda,\mu,\eta)$ with their
$\eta$ components closest to $\eta_{\rm tar}$ are listed.
\label{tab:ell_avg}}
\begin{center}
{\footnotesize
\begin{tabular}{rcrcrrrrrrr}
 \hline
 & & \multicolumn{3}{c}{$\#$ Computed eigenvalues} & \multicolumn{3}{c}{$\#$ Subspace updates} & \multicolumn{3}{c}{Time (seconds)} \\
$\eta_{\rm tar}$ & $\#$ targeted & average & min & max & average & min & max & average & min & max \\
\hline \rule{0pt}{2.3ex}%
0 & 40 & 40 & 40 & 40 & 9.3 & 9 & 10 & 3.9 & 2.7 & 7.1 \cr
200 & 20 & 33.8 & 24 & 42 & 199.8 & 121 & 311 & 159.9 & 64.5 & 285.1\cr
1000 & 10 & 118.1 & 51 & 170 & 871.2 & 292 & 1354 & 771.3 & 264.8 & 1246.5\cr
 \hline
\vspace{-2em}
\end{tabular}}
\end{center}
\end{table}

\medskip

In Table \ref{tab:ellipsoid}, we provide the three computed eigenvalues closest to the target for cases
a), b), c). While the closest three eigenvalues for case a) have already been listed in \cite{Calin3},
it was not possible then to compute accurate solutions for cases b) and c) as this requires matrices
larger than the methods at that time could handle.

\begin{table}[htb]
\caption{A list of three eigenvalues $(\lambda,\mu,\eta)$ with their $\eta$ components closest to the targets
$\eta_{\rm tar}=0$, $\eta_{\rm tar}=200$, $\eta_{\rm tar}=1000$ for the 3-parameter eigenvalue problem
resulting from the ellipsoidal wave equation \eqref{eq:main_ODE_sys}
with the radii values $x_0 = 1, y_0 = 1.5$, $z_0 = 2$ and for the configuration $(\rho,\sigma,\tau)=(0,0,0)$.
The eigenfrequencies $\omega$ corresponding to these computed eigenvalues are also listed in the last column.
\label{tab:ellipsoid}}
\begin{center}
{\footnotesize
\begin{tabular}{rrrrr}
 \hline
Target $\eta_{\rm tar}$ & \multicolumn{1}{c}{$\lambda$} & \multicolumn{1}{c}{$\mu$} & \multicolumn{1}{c}{$\eta$} & \multicolumn{1}{c}{$\omega$} \\
\hline \rule{0pt}{2.3ex}%
 & \phantom{111}0.84989209 & \phantom{111}$-3.75231782$ & \phantom{111}2.40498182 & \phantom{1}2.34458979\cr
0 & \phantom{111}7.22643744 & \phantom{11}$-13.03122756$ & \phantom{111}5.59866649 & \phantom{1}3.57728277\cr
 & \phantom{111}2.05458475 & \phantom{11}$-13.46994828$ & \phantom{111}7.46473320 & \phantom{1}4.13064732\cr
\hline & \phantom{1}141.38925861& \phantom{1}$-404.17476271$ & \phantom{1}200.60583308 & 21.41325801\cr
200 & \phantom{11}63.08832970 & \phantom{1}$-423.06537129$ & \phantom{1}199.27518005 & 21.34212093 \cr
 & \phantom{1}317.06224687 & \phantom{1}$-551.46171960$ & \phantom{1}201.03180983 & 21.43598096 \cr
\hline & 1413.79140334 & $-2535.12357474$ & \phantom{1}999.75548115 & 47.80329890 \cr
1000 & \phantom{1}366.26819031 & $-2143.09786044$ & 1000.47359673 & 47.82046416 \cr
 & 1725.45584215 & $-2758.97471801$ & \phantom{1}999.44259731 & 47.79581804 \\
 \hline
\vspace{-1.5em}
\end{tabular}}
\end{center}
\end{table}
}
\end{example}
\smallskip

\begin{example}[Subspace Iteration on the Ellipsoidal Wave Equation]\label{ex:ell2}
{\rm
We aim to compute the same eigenvalues as in the previous example using
Algorithm~\ref{alg3}.
Initially we choose the search space of dimension $\ell=6$ and apply zero Arnoldi steps (in
cases a) and b)) or one Arnoldi step (in case c)) in the expansion.
We apply
SVD filtering to $F_j$ as in line~\ref{lin2:f}, where we set the cutoff parameter $\zeta=10^{-5}$.
The dimension of
${\rm span}(Q_1\otimes Q_2\otimes Q_3)$ in line~\ref{lin2:shrink} is limited to 1000, 5000 and 15000
in cases a), b) and c), respectively. Smaller subspace dimensions are sufficient when eigenvalues
with smaller $\eta$ components are targeted as these eigenvalues lie in the exterior of the
spectrum. On the other hand, larger subspaces are needed for larger values of $\eta_{\rm tar}$.
In every iteration we compute
100 (in
cases a) and b)) or 50 (in case c)) Ritz values of the projected 3-parameter eigenvalue problem in line~\ref{lin2:extract}
closest to the prescribed target. This is followed by one step (in
cases a) and b)) or three steps (in case c)) of the TRQI to refine each Ritz pair in line~\ref{lin2:refine}.
After that, we consider a Ritz pair
 as a candidate for an eigenpair if its residual is smaller then $\delta=10^{-2}$, and
if it satisfies the selection criterion \eqref{eq:selcrit} with $\xi_1=10^{-1}$.
In this case, we refine the Ritz pair with up to 3 additional steps of
the TRQI. The final residuals corresponding to the Ritz pairs are accepted
small enough with the particular choices of the parameters as in the Jacobi--Davidson method,
that is $\varepsilon=10^{-8}$ and $\xi_2=10^{-4}$ in line~\ref{lin2:convergence}.

We computed 80 eigenvalues for cases a), b), c) from Example \ref{ex:ell1}.
Computational times for cases a), b), c) are 8, 60, 421 seconds,
and 3, 3, 5 subspace iterations have been carried out, respectively.
Figure \ref{fig:ellsi1} shows the values of $|\eta-\eta_{\rm tar}|$ for the computed eigenvalues with
respect to their order of retrieval; one can observe a behavior similar to the Jacobi--Davidson method,
i.e., for smaller values of $\eta_{\rm tar}$ it is possible to observe a monotonicity in $|\eta-\eta_{\rm tar}|$
relative to the order of the retrieval of the eigenvalues, which gradually degrades as $\eta_{\rm tar}$
is increased.

\begin{figure}[htb]
\begin{center}
\includegraphics[scale=0.047]{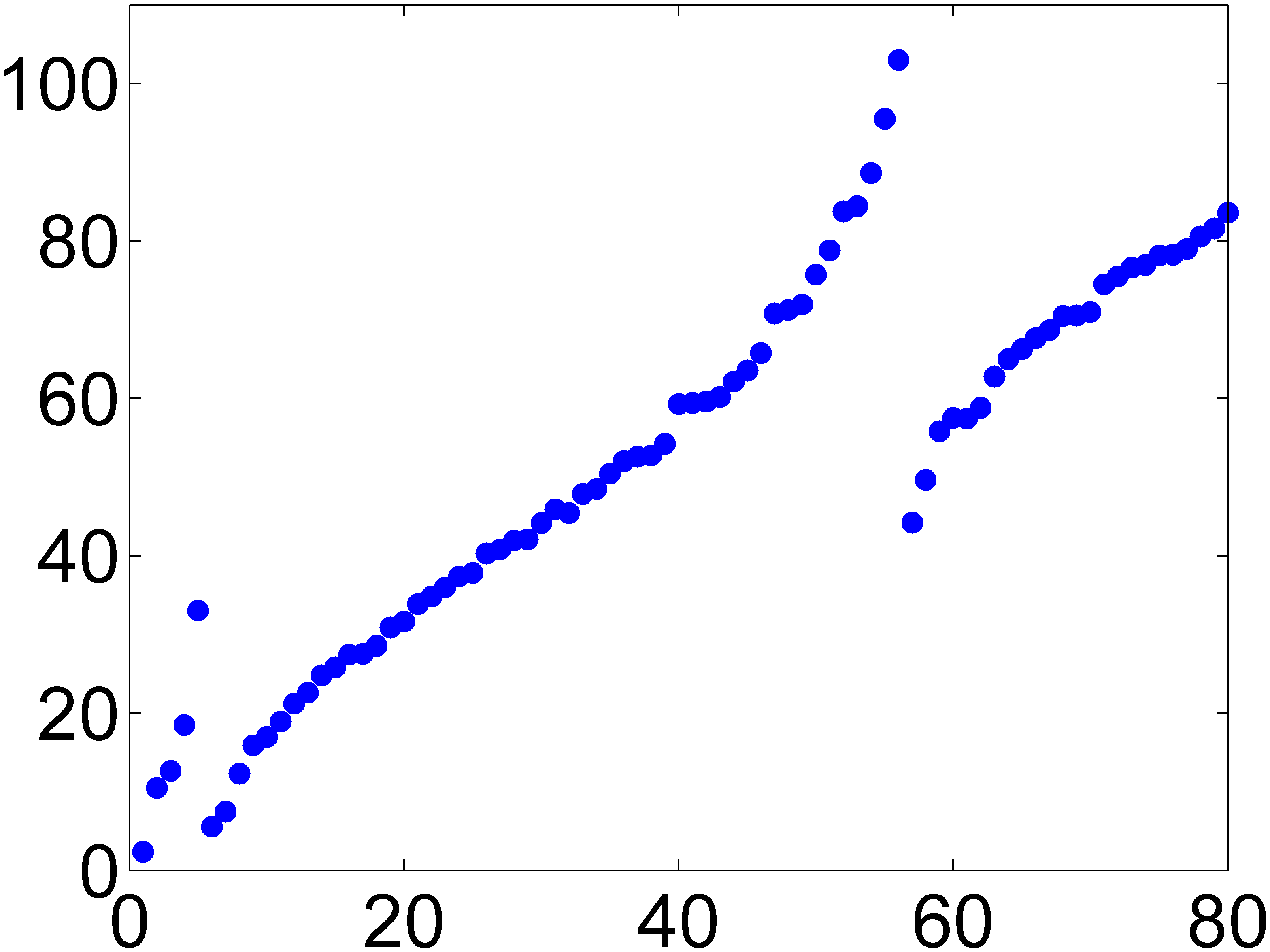} \
\includegraphics[scale=0.047]{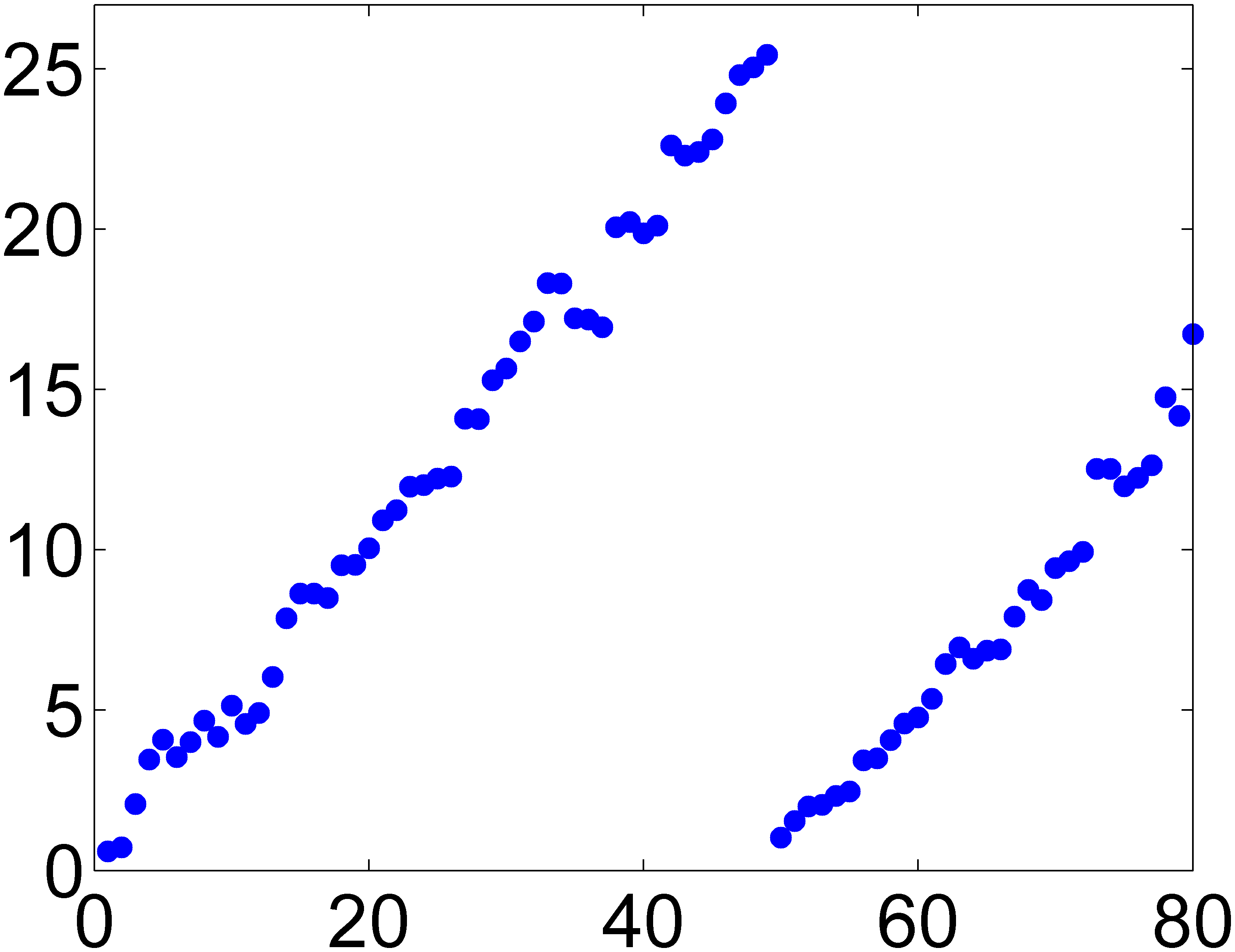} \
\includegraphics[scale=0.047]{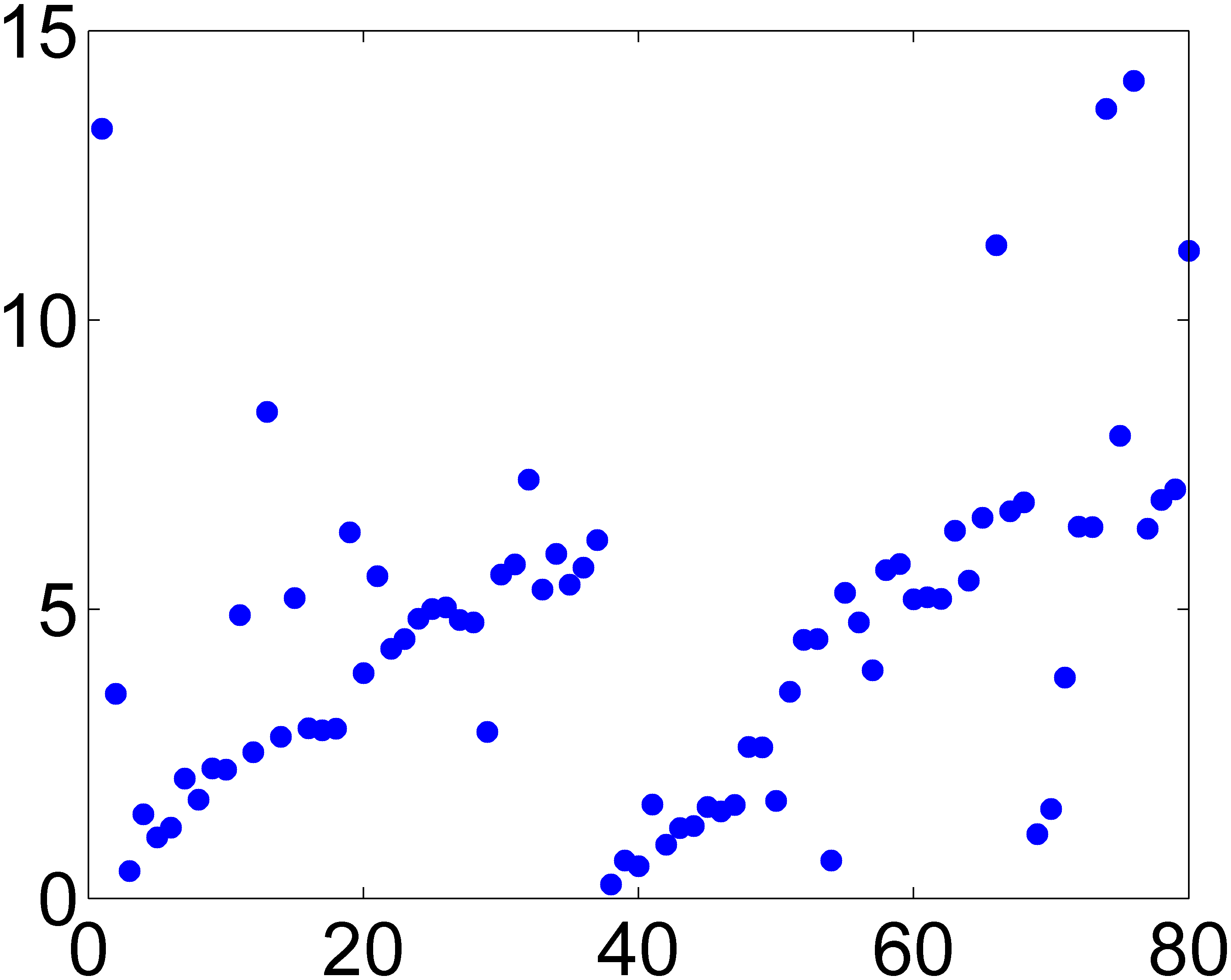}
\end{center}
\vspace{-0.7em}
\caption{Application of the subspace iteration with Arnoldi expansion to the 3-parameter
eigenvalue problem in Section~\ref{sec:result_ellip_wave}.
The values $|\eta-\eta_{\rm tar}|$ (vertical axis) of the first 80 computed eigenvalues $(\lambda,\mu,\eta)$
are plotted with respect to their order of retrieval (horizontal axis) for case
a) $\eta_{\rm tar}=0$ (left), b) $\eta_{\rm tar}=200$ (middle), and c) $\eta_{\rm tar}=1000$ (right).}
\label{fig:ellsi1}
\end{figure}

The SVD filtering does not reduce the dimension of the subspaces enough, so we also have to perform the shrinking
in line~\ref{lin2:shrink} of Algorithm~\ref{alg3}.
A stricter SVD filtering with a larger cutoff
is not a solution, as this results
in the removal of some of the good search spaces.
For case b) the dimensions of ${\rm span} ( Q_1\otimes Q_2\otimes Q_3 )$
are 154548, 48300, 11340 in the first, second, third iterations, respectively,
which are all shrunk into subspaces of dimension smaller than 5000.
The appearance of larger subspaces in the initial iterations is typical.
In the first iteration, the Arnoldi expansion increases the dimension of the search
space considerably, but, after a few subspace iterations, the search space contains good approximations
of the eigenvectors, and the Arnoldi expansion does not yield many independent directions.
These findings are in line with results obtained for the 2-parameter case \cite{mesp10}.

Following the practice in Example \ref{ex:ell1}, we  explored how many eigenvalues need to be
computed in total with the above settings in order to retrieve all of the 40, 20 and 10 eigenvalues with $\eta$ components
closest to $\eta_{\rm tar}$ for cases a), b) and c), respectively. The results are reported in Table \ref{tab:ell_avg_si}.

\begin{table}[htb]
\caption{This table concerns the application of subspace iteration with Arnoldi expansion
to the 3-parameter eigenvalue problem in Section~\ref{sec:result_ellip_wave}.
Total number of eigenvalues that had to be computed, subspace iterations that had to be performed and computational times
in order to retrieve the targeted number of eigenvalues $(\lambda,\mu,\eta)$ with $\eta$ components
closest to $\eta_{\rm tar}$ are listed.
\label{tab:ell_avg_si}}
\begin{center}
{\footnotesize
\begin{tabular}{rcccrcccrrr}
 \hline
 & & \multicolumn{3}{c}{$\#$ Computed eigenvalues} & \multicolumn{3}{c}{$\#$ Subspace iterations} & \multicolumn{3}{c}{Time (seconds)} \\
$\eta_{\rm tar}$ & $\#$ targeted & average & min & max & average & min & max & average & min & max \\
\hline \rule{0pt}{2.3ex}%
0 & 40 & 58.2 & 40 & 82 & 2.7 & 2 & 3 & 7.7 & 5.5 & 10.3 \cr
200 & 20 & 74.1 & 40 & 109 & 3.5 & 3 & 4 & 59.8 & 50.8 & 69.2 \cr
1000 & 10 & 68.9 & 50 & 88 & 5.3 & 4 & 7 & 461.2 & 307.4 & 651.7\cr
 \hline
\end{tabular}}
\end{center}
\end{table}
}
\end{example}

When we compare the numerical results obtained for the Jacobi--Davidson method and the subspace iteration method,
we see that the subspace iteration works slightly faster for mildly interior eigenvalues.
This comes at the expense of much larger memory requirements; 
for instance,
at least 16 GB of RAM is needed by subspace iteration to use a search space of dimension 15000.
If subspaces are restricted to small dimensions, then we do not get approximations that are good enough
to lead to eigenpairs (even if additional subspace iterations are allowed).

\subsection{Baer wave equations}\label{sec:result_Baer_wave}
By solving the 3-parameter eigenvalue problem resulting from the Baer wave equations discussed in Section~\ref{sec:motive_Baer_wave},
we can obtain many estimates for low eigenfrequencies of the Helmholtz equation (\ref{eq:helm}) on the specified intersection of paraboloids.
We could not find any similar numerical results regarding this example in the literature, so, up to our knowledge, this is the first time that
Helmholtz equation is solved numerically in paraboloidal coordinates. The results could be used for future comparisons to other
numerical methods.

As in the previous subsection we discretize the system of Baer wave equations \eqref{eq:baerF}
for the configuration $(\rho,\sigma)=(0,0)$ with Chebyshev collocation on 300 points.
We are interested in
\begin{enumerate}
\item[a)] the lowest eigenfrequencies (i.e., $\eta_{\rm tar}=0$), and
\item[b)] the eigenfrequencies closest to 10 (i.e., $\eta_{\rm tar}=100$).
\end{enumerate}
In case b), we apply the substitution
$
	(\widetilde{\lambda}, \widetilde{\mu}, \widetilde{\eta})
			=
	(\lambda,\mu,\eta - \eta_{\rm tar})$
and search for eigenvalues close to $\widetilde \eta=0$ of the transformed problem,
with the coefficient matrices
 $\widetilde A_j=A_j-\eta_{\rm tar} D_j$, $\widetilde B_j=B_j$, $\widetilde C_j=C_j$, and $\widetilde D_j=D_j$
for $j = 1,2,3$. Once again, before applying the numerical methods,
we multiply the $j$th equation by $\widetilde A_j^{-1}$ for $j = 1,2,3$.
\smallskip

\begin{example}[Results for Baer wave equations]\label{ex:baer}
{\rm We apply both
algorithms to the problem above. Using the same settings as in Example \ref{ex:ell1}, the Jacobi--Davidson method
 computes 80 eigenvalues in 10 seconds after 16 subspace updates in case a), and
 in 314 seconds using 341 subspace updates in case b). For subspace iteration,
 we use the same settings as in cases a) and b) of Example \ref{ex:ell2}.
This means that we limit the dimension of the search space to 1000
in case a), and 5000 in case b). The method requires 7 seconds and 2 subspace iterations to compute 80 eigenvalues
in case a), and 59 seconds and 3 subspace iterations in case b).

We omit the plots of $|\eta - \eta_{\rm tar}|$ with respect to the retrieval order
for the converged eigenvalues, as they turn out to be similar to the left-hand and the middle plots
in Figures \ref{fig:elljd1} and \ref{fig:ellsi1}. As in Examples \ref{ex:ell1} and \ref{ex:ell2},
we end up computing more eigenvalues for larger values of $\eta_{\rm tar}$ in order to
retrieve all of the desired eigenvalues closest to $\eta_{\rm tar}$.

Similar to the previous examples, we tested how many eigenvalues need to be computed
with the settings above in order to retrieve all of the 40 and 20 eigenvalues with $\eta$ components
closest to $\eta_{\rm tar}$ for cases a) and b), respectively. For both methods, Table \ref{tab:baer_avg_si}
reports the average results together with the best and worst run over a set of 10 different random
initial subspaces.

\begin{table}[htb]
\caption{The Jacobi--Davidson method (JD) and the subspace
iteration with Arnoldi expansion (SI) applied to the 3-parameter eigenvalue problem
in Subsection~\ref{sec:result_Baer_wave}. This table lists the total number of eigenvalues that had
to be computed, number of subspace iterations that had to be performed and computational times
required to retrieve all of the targeted eigenvalues $(\lambda,\mu,\eta)$ with their $\eta$
components closest to $\eta_{\rm tar}$.
\label{tab:baer_avg_si}}
\begin{center}
{\footnotesize
\begin{tabular}{crcrcrrrrrrr}
 \hline
 & & & \multicolumn{3}{c}{$\#$ Computed eigenvalues} & \multicolumn{3}{c}{$\#$ Subspace iterations}
 & \multicolumn{3}{c}{Time (seconds)} \\
method & $\eta_{\rm tar}$ & $\#$ targeted & average & min & max & average & min & max & average & min & max \\
\hline \rule{0pt}{2.3ex}%
JD & 0 & 40 & 40 & 40 & 40 & 8.8 & 7 & 10 & 3.4 & 1.8 & 7.6 \cr
JD & 100 & 20 & 70.6 & 41 & 142 & 268.5 & 137 & 453 & 209.7 & 104.0 & 356.0 \cr
\hline \rule{0pt}{2.3ex}%
SI & 0 & 40 & 44.0 & 40 & 80 & 2.1 & 2 & 3 & 5.9 & 5.4 & 9.4 \cr
SI & 100 & 20 & 88.6 & 63 & 104 & 4.3 & 4 & 5 & 74.2 & 68.0 & 86.2\cr
 \hline
\end{tabular}}
\end{center}
\end{table}

Algorithms \ref{alg:jd} and \ref{alg3} return the same 10 eigenfrequencies
closest to 10. In particular, the results by both algorithms
agree on the first three eigenfrequencies larger than 10; these eigenfrequencies
are listed in Table \ref{tab:paraboloid} along with the lowest six eigenfrequencies from case a).
We verify the correctness of the computed results by means of the Klein oscillation
property, which concerns the number of zeros of $X_i(\xi_i)$ as in (\ref{eq:baer}).
This property is formally stated in the next theorem. To our knowledge, it has not been
explicitly shown for the system of Baer wave equations up to this point, so a proof is
included in Appendix \ref{proof_Klein_prop}.
 \begin{theorem}\label{thm:baer}
For each of the four possible
configurations $(\sigma,\tau)$ in \eqref{eq:baerXF},
the system of Baer wave differential equations \eqref{eq:baer} has the Klein
oscillation property, i.e., all of its eigenvalues are real and for each triple of
nonnegative integers $(j_1,j_2,j_3)$ there exists exactly one eigenvalue $(\lambda,\mu,\eta)$ such
that the corresponding eigenfunctions $X_1(\xi_1)$, $X_2(\xi_2)$, $X_3(\xi_3)$ have exactly $j_1$
zeros on $(\gamma,c)$, $j_2$ zeros on $(c,b)$, and $j_3$ zeros on $(b,\beta)$, respectively.
\end{theorem}

In Table \ref{tab:paraboloid}, we provide an integer triple $(j_1,j_2,j_3)$ for each eigenfrequency $\omega$
with $j_i$ denoting the index of $X_i(\xi_i)$ as in (\ref{eq:baer}), that is the number of the zeros of the
corresponding solution $X_i(\xi_i)$ on the interval $(\ell_i, \ell_{i+1})$ with $\ell_1 = \gamma = 0$,
$\ell_2 = c = 1$, $\ell_3 = b = 3$, $\ell_4 = \beta = 5$.
The reported results in the table are in harmony with Theorem~\ref{thm:baer},
that is there exists exactly one eigenvalue corresponding to each nonnegative triple $(j_1,j_2,j_3)$.
Furthermore, the results confirm that the lowest eigenfrequencies have the smallest indices, as expected
in theory \cite[Section 8]{Atkinson2}.

\begin{table}[htb]
\caption{
Results for the Helmholtz equation
with a Dirichlet boundary condition on a domain bounded by two elliptic paraboloids
$\gamma=0$ and $\beta=5$ in paraboloidal coordinates with $c = 1$ and $b = 3$
for the configuration $(\sigma,\rho)=(0,0)$.
Estimates for the lowest 6 eigenfrequencies and the first 3 eigenfrequencies larger
than 10 of a related 3-parameter eigenvalue problem, namely the Baer wave equations
(\ref{eq:baer}), are listed in the table. In each row, in addition to the eigenfrequency $\omega$,
the corresponding eigenvalue $(\lambda,\mu,\eta)$ and the indices $(j_1,j_2,j_3)$ of the
corresponding functions $X_1,X_2,X_3$ are also listed.
\label{tab:paraboloid}}
{\footnotesize
\begin{center}
\begin{tabular}{rrrrccc}
 \hline
\multicolumn{1}{c}{$\lambda$} & \multicolumn{1}{c}{$\mu$} & \multicolumn{1}{c}{$\eta$} & \multicolumn{1}{c}{$\omega$} & $j_1$ & $j_2$ & $j_3$ \\
\hline \rule{0pt}{2.3ex}%
 \phantom{1}4.68572309 & $\phantom{1}-4.68336498$ & 1.06171767 & 1.03039685 & 0 & 0 & 0\\
 \phantom{1}8.98735825 & $-10.98752097$ & 2.52640136 & 1.58946575 & 0 & 1 & 0\\
 \phantom{1}7.84880354 & $\phantom{1}-9.81384367$ & 2.70641882 & 1.64511970 & 0 & 0 & 1\\
 23.88802753 & $-18.11389297$ & 3.33102584 & 1.82510982 & 1 & 0 & 0\\
 15.35149716 & $-20.44266626$ & 4.60326049 & 2.14552103 & 0 & 2 & 0\\
 13.98083910 & $-19.03124115$ & 4.90993954 & 2.21583834 & 0 & 1 & 1\\
\multicolumn{1}{c}{$\vdots$} & \multicolumn{1}{c}{$\vdots$} &
\multicolumn{1}{c}{$\vdots$} & \multicolumn{1}{c}{$\vdots$} &
\multicolumn{1}{r}{$\vdots\,$} & \multicolumn{1}{r}{$\vdots\,$} & \multicolumn{1}{r}{$\vdots\,$} \\
368.61672638 & $-467.93904610$ & 100.12807872 & 10.00640189 & 3 & 10 & 2 \\
909.43143081 & $-643.56267025$ & 100.20818157 & 10.01040367 & 9 & 4 & 0 \\
315.21740925 & $-436.37381658$ & 100.32096431 & 10.01603536 & 2 & 10 & 3 \\
\hline
\end{tabular}
\end{center}}
\end{table}

\medskip

Note that we can approximate the solutions $X_1(\xi_1)$, $X_2(\xi_2)$, $X_3(\xi_3)$ of the
Baer wave equation by employing \eqref{eq:baerF} subject to the boundary conditions \eqref{baer:bc},
as well as eigenvectors of the discretized algebraic 3-parameter eigenvalue problem. We can
combine them in a smooth eigenfunction $X(\xi)$ bounded at the points $\xi=1$, $\xi=3$ and satisfying
\[
 (\xi-1)(\xi-3)\,X'' + \tfrac{1}{2}(2\xi-4)\,X' + (\lambda +\mu \xi +\eta \xi^2)\,X=0
\]
over $\xi \in [0,5]$ subject to $X(0)=0$, $X(5)=0$. The eigenfunctions corresponding to
the six lowest eigenfrequencies computed are displayed in Figure \ref{fig:baer6}.

\begin{figure}[htb]
\begin{center}
\includegraphics[scale=0.038]{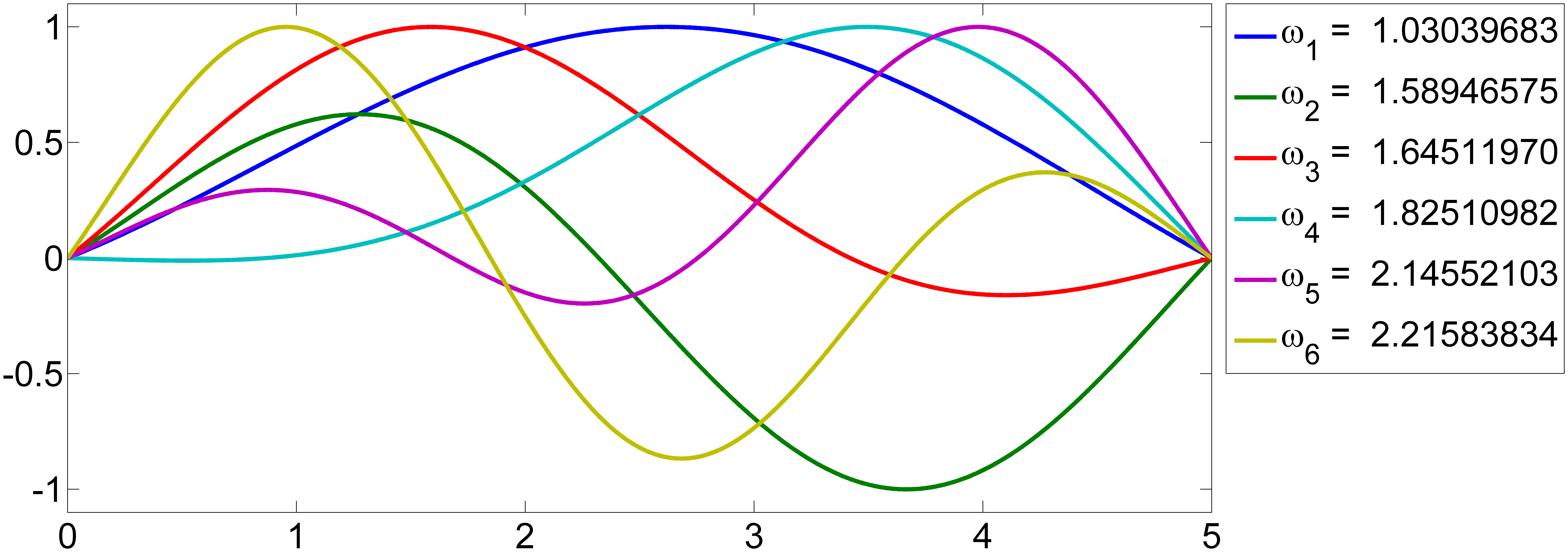}
\end{center}
\vspace{-0.5em}
\caption{The eigenfunctions corresponding to the six lowest eigenfrequencies
in Table \ref{tab:paraboloid}.}
\label{fig:baer6}
\vspace{-1.5em}
\end{figure}
}
\end{example}

\subsection{Randomly generated example}\label{sec:result_random_example}
Our final example is a 3-parameter eigenvalue problem generated in Matlab in such a way that we know
all of the eigenvalues. We first form the matrices
\[
 A_i=U_i\,{\rm diag}(a_i)\,V_i,\quad
 B_i=U_i\,{\rm diag}(b_i)\,V_i,\quad
 C_i=U_i\,{\rm diag}(c_i)\,V_i,\quad
 D_i=U_i\,{\rm diag}(d_i)\,V_i,
\]
where $a_i,b_i,c_i,d_i$ are real random vectors of size $n$, and $U_i$, $V_i$ are random well-conditioned sparse
matrices of size $n\times n$ for $i=1,2,3$. Observe that the eigenvalues of the resulting 3-parameter eigenvalue
problem are the solutions to the $3\times 3$ linear systems
\begin{align*}
(a_1)_{\ell}&=\lambda\,(b_1)_{\ell} + \mu\,(c_1)_{\ell} + \eta\,(d_1)_{\ell},\\
(a_2)_{j}&=\lambda\,(b_2)_{j} + \mu\,(c_2){j} + \eta\,(d_2)_{j},\\
(a_3)_{k}&=\lambda\,(b_3)_{k} + \mu\,(c_3)_{k} + \eta\,(d_3)_{k}
\end{align*}
for
$\; \ell,j,k=1,\ldots,n$, where $(a_i)_p, (b_i)_p, (c_i)_p,(d_i)_p$ denote the $p$th entries
of $a_i, b_i, c_i,d_i$ for $i = 1,2,3$. This observation enables us to compute all $n^3$ eigenvalues for moderate
values of $n$, e.g., $n=100$.
\smallskip

\begin{example}[Randomly generated example]\label{ex:random}
{\rm
We set $n=100$ and generate entries of $a_i,b_i,c_i,d_i$ randomly, by first selecting them independently
from a uniform distribution over $[0,1]$ and then applying shifts. The Matlab code generating $a_i, b_i, c_i,d_i$ and the
 matrices $U_i, V_i$ for $i = 1,2,3$ is given below. \smallskip

{\small\begin{verbatim}
U1 = 0.3*sprand(n,n,0.04)+speye(n); U2 = 0.3*sprand(n,n,0.04)+speye(n);
U3 = 0.3*sprand(n,n,0.04)+speye(n); V1 = 0.3*sprand(n,n,0.04)+speye(n);
V2 = 0.3*sprand(n,n,0.04)+speye(n); V3 = 0.3*sprand(n,n,0.04)+speye(n);
a1 = rand(n,1)-0.5; b1 = rand(n,1)+2; c1 = rand(n,1); d1 = rand(n,1)-1;
a2 = rand(n,1)-0.5; b2 = rand(n,1); c2 = rand(n,1)+2; d2 = rand(n,1)+0.5;
a3 = rand(n,1)-0.5; b3 = rand(n,1)-1; c3 = rand(n,1); d3 = rand(n,1)+2;
\end{verbatim}}\smallskip
\noindent
Figure \ref{fig:random13} illustrates the resulting eigenvalues projected orthogonally onto the plane $\mu = 0$.
Orthogonal projections of the eigenvalues onto the planes $\lambda=0$ and $\eta=0$ yield similar pictures.

\begin{figure}[htb]
\begin{center}
\includegraphics[scale=0.06]{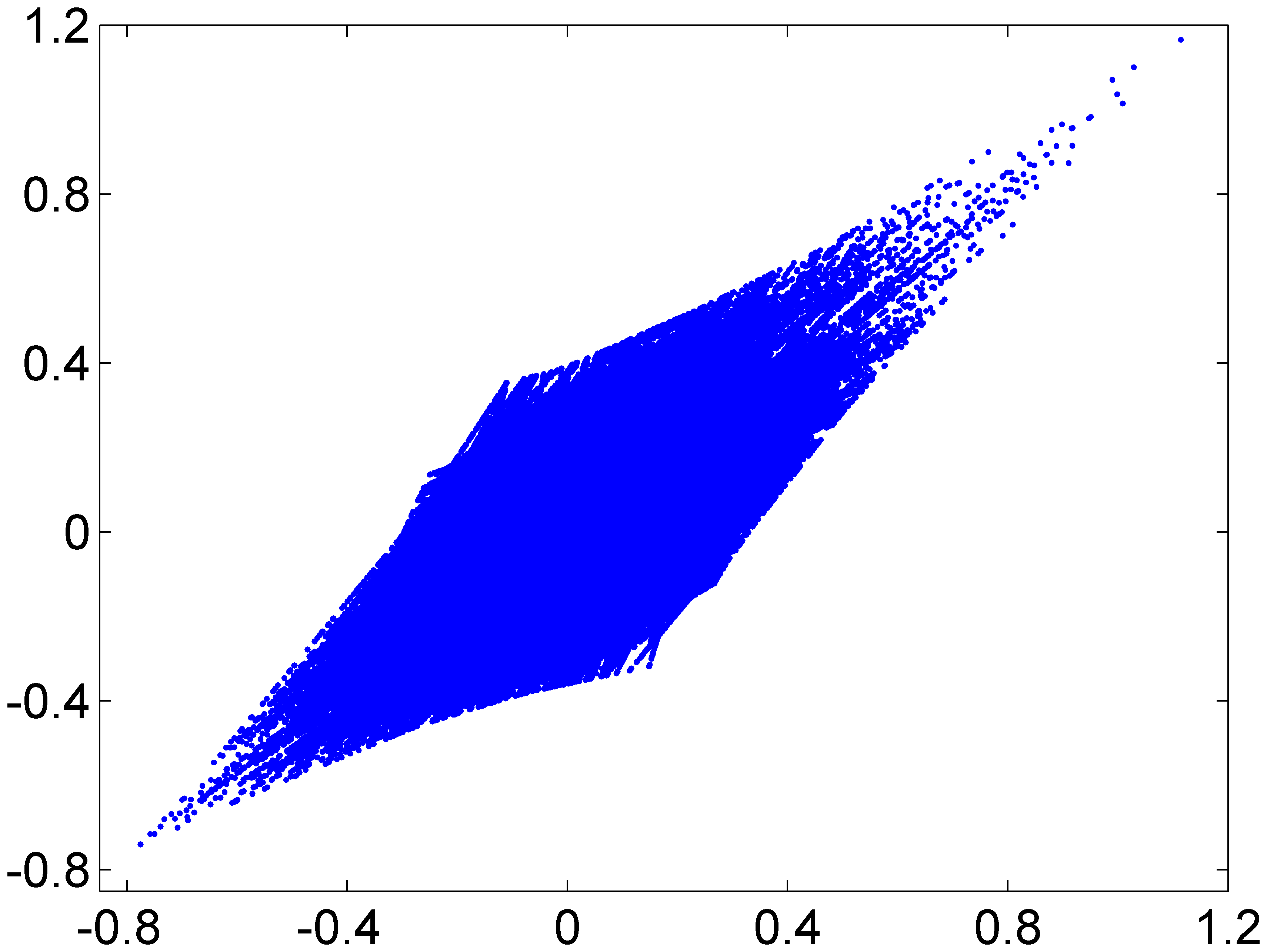}
\end{center}
\vspace{-0.5em}
\caption{The orthogonal projections of the eigenvalues of the 3-parameter eigenvalue problem considered in Example \ref{ex:random}
onto the plane $\mu=0$. 
The horizontal and vertical axes correspond to the $\lambda$ and $\eta$ components, respectively.\\[1.0em]}
\label{fig:random13}
\vspace{-1,5em}
\end{figure}

We test the methods to compute
\begin{enumerate}
\item[a)] 20 external eigenvalues with $\eta$ closest to $\eta_{\rm tar}=-0.8$, and
\item[b)] 10 mildly interior eigenvalues with $\eta$ closest to $\eta_{\rm tar}=-0.5$.
\end{enumerate}
Following the practice in the other examples,
we use the substitution
$
	(\widetilde{\lambda},\widetilde{\mu},\widetilde{\eta})
			=
	(\lambda,\mu,\eta-\eta_{\rm tar})$
and search for the eigenvalues
of the transformed problem having $|\widetilde{\eta}|$
 as small as possible.

We apply the Jacobi--Davidson method, where we solve the correction equation exactly,
and use up to 3 TRQI steps as well as the choices $\delta=10^{-6}$ and $\varepsilon=10^{-10}$ to decide whether
the residual of a Ritz pair is small enough to consider it as an eigenpair. All of the remaining parameters are as
in the previous examples. The results are presented in Table \ref{tab:ell_avg_rnd}.

\begin{table}[htb]
\caption{The Jacobi--Davidson method applied to a random 3-parameter eigenvalue problem in Example \ref{ex:random}.
The table reports the number of eigenvalues that had to be computed, subspace iterations that had to be
performed and computational times required in order to retrieve all of the targeted eigenvalues
$(\lambda,\mu,\eta)$ with $\eta$ components closest to $\eta_{\rm tar}$.
\label{tab:ell_avg_rnd}}
\begin{center}
{\footnotesize
\begin{tabular}{ccccccccccc}
 \hline
 & & \multicolumn{3}{c}{$\#$ Computed eigenvalues} & \multicolumn{3}{c}{$\#$ Subspace updates} & \multicolumn{3}{c}{Time (seconds)} \\
$\eta_{\rm tar}$ & $\#$ targeted & average & min & max & average & min & max & average & min & max \\
\hline \rule{0pt}{2.3ex}%
$-0.8$ & 20 &  78.5 & 38 & 137 & 206.3 &  112 &  363 & 197 &  102 &  346 \cr
$-0.5$ & 10 & 97.3 & 45 & 195 & 284.4 & 148 & 478 & 266 & 135 & 454 \cr
 \hline
\end{tabular}}
\end{center}
\end{table}

\medskip

The subspace iteration does not work well on this example. We 
could not find a combination of parameters to make it competitive with the Jacobi--Davidson method.
The method computes some eigenpairs, but requires a lot of time and returns many eigenvalues
 far away from the target.

It was not possible to compute the eigenvalues $(\lambda,\mu,\eta)$ with the minimal values of $|\eta|$ by
Algorithm~\ref{alg:jd} and Algorithm~\ref{alg3}. The difficulty is that these eigenvalues are highly interior.
If, instead, we aim for the eigenvalues closest to $(0,0,0)$, then the Jacobi--Davidson method performs well
with the parameter values indicated above but by solving the correction equations approximately,
in particular by employing 10 steps of GMRES with $A_i^{-1}$ as the preconditioner for the $i$th equation for $i = 1,2,3$.
The method converges to 50 eigenvalues after 119 subspace updates in 82 seconds. All but three of the 50 eigenvalues
closest to $(0,0,0)$ are among the converged eigenvalues and the remaining eigenvalues converged after a few more
iterations. This shows that the Jacobi--Davidson method is capable of locating the eigenvalues closest to a prescribed
point, even if these eigenvalues are interior ones.
}
\end{example}

\section{Concluding Remarks}

We have introduced a Jacobi--Davidson method (Algorithm~\ref{alg:jd}) and a subspace iteration method (Algorithm~\ref{alg3})
that restarts the subspace at every iteration for the 3-parameter eigenvalue problem.
Matlab implementations are available in package \texttt{MultiParEig} \cite{BorMC1}.
The Jacobi--Davidson method is especially well-suited to locate eigenvalues close to a prescribed target.
This method seems to perform well in practice also to locate eigenvalues $(\lambda,\mu,\eta)$ whose $\eta$ components
are close to a prescribed target, while the proposed subspace iteration method is specifically designed for this task.
Numerical experiments indicate that when the eigenvalues are targeted based on their $\eta$ components,
both methods are very good at locating exterior eigenvalues and mildly interior eigenvalues, but both methods
struggle to compute interior eigenvalues.

Based on the numerical experiments, it is not possible to draw a clear conclusion
regarding the efficiency of the methods in comparison to each other. In some of the numerical results reported, the Jacobi--Davidson
method exhibits better performance in terms of efficiency, in others the subspace iteration method appears better. To this end,
the choice of the parameters, such as the thresholds for the residuals of the Ritz pairs and maximal subspace dimensions,
plays an important role.

%

\section{Acknowledgement}

The authors are grateful to two anonymous referees and the associate editor in charge of the manuscript  
for their time and valuable suggestions on an initial version of this manuscript. There are no conflicts of interest to this work.

\medskip

\appendix

\section{Proof of Theorem \ref{thm:baer}}\label{proof_Klein_prop}
We will only consider the configuration $(\rho,\sigma)=(0,0)$, as the other
three configurations can be treated similarly. Inspired by \cite{Cohl}, we introduce
\[
g(z):= \left| (\xi-b)(\xi-c) \right|^{1/2}.
\]
We can now write \eqref{eq:baer} as a 3-parameter Sturm--Liouville eigenvalue problem
\begin{align*}
	\big(g(\xi_1)X_1'\big)'&+\frac{1}{g(\xi_1)}
	(\lambda +\mu \xi_1 +\eta \xi_1^2)\,X_1=0, \quad \gamma<\xi_1<c,\\
	\big(g(\xi_2)X_2'\big)'&+\frac{1}{g(\xi_2)}
	(\lambda +\mu \xi_2 +\eta \xi_2^2)\,X_2=0, \quad c<\xi_2<b,\\
	\big(g(\xi_3)X_3'\big)'&+\frac{1}{g(\xi_3)}
	(\lambda +\mu \xi_3 +\eta \xi_3^2)\,X_3=0, \quad b<\xi_3<\beta.
\end{align*}
Next, we introduce the elliptic integral
\[
G(s):=\int_\gamma^s\frac{d\sigma}{g(\sigma)},
\]
which is an increasing absolutely continuous function,
and apply the substitution $t_i=G(\xi_i)$, $u_i(t_i)=X_i(\xi_i)$
for $i=1,2,3$. This gives rise to
\begin{align}
u_1''&+(\lambda +\mu \phi(t_1) +\eta \phi(t_1)^2)\,u_1=0, \quad T_0<t_1<T_1,\nonumber\\
u_2''&-(\lambda +\mu \phi(t_2) +\eta \phi(t_2)^2)\,u_2=0, \quad T_1<t_2<T_2,\label{eq:baersub}\\
u_3''&+(\lambda +\mu \phi(t_3) +\eta \phi(t_3)^2)\,u_3=0, \quad T_2<t_3<T_3,\nonumber
\end{align}
where $T_0=G(\gamma)=0$, $T_1=G(c)$, $T_2=G(b)$, $T_3=G(\beta)$,
and $\phi:[T_0,T_3]\to[\gamma,\beta]$ is the inverse function of $G$.
It can be shown that \eqref{eq:baersub} is a right definite problem due to \cite[Thm.~3.6.2]{HV}.
Specifically, let us consider the corresponding determinant function (see, e.g., \cite{Atkinson2}) given by
\begin{align*}
\delta_0(t_1,t_2,t_3) & \; = \;
\left|
\begin{array}{rrr}
1 & \phi(t_1) & \phi(t_1)^2\cr
-1 & -\phi(t_2) & -\phi(t_2)^2\cr
1 & \phi(t_3) & \phi(t_3)^2\cr
\end{array}\right|\\[0.3em]
& \; = \; (-1)(\phi(t_2)-\phi(t_1))(\phi(t_3)-\phi(t_1))(\phi(t_3)-\phi(t_2)),
\end{align*}
where $T_0\le t_1\le T_1\le t_2\le T_2\le t_3\le T_3$.
One can verify that $\delta_0(t_1,t_2,t_3)<0$ for all $t_1<t_2<t_3$. Since
$\delta_0$ is of constant sign on a dense subset of $[T_0,T_1]\times[T_1,T_2]\times[T_2,T_3]$, it follows from
\cite[Thm.~3.6.2]{HV}) that the problem is right definite. Hence, \cite[Thms.~3.5.1 and 3.5.2]{HV}
imply that the Klein oscillation theory holds for the problem. In particular, all eigenvalues are real and
for each triple of nonnegative integers $(j_1,j_2,j_3)$, there exists exactly one eigenvalue $(\lambda,\mu,\eta)$
such that the corresponding eigenfunction $u_i(t_i)$ has exactly $j_i$ zeros on $(T_{i-1},T_i)$ for $i=1,2,3$.


\end{document}